\documentclass[12pt]{amsart}
\usepackage[active]{srcltx}
\usepackage{a4wide}
\usepackage{amsthm,amsfonts,amsmath,mathrsfs,amssymb}
\usepackage{dsfont}
\usepackage[T1]{fontenc}
\usepackage[utf8]{inputenc}
\usepackage{fourier}
\usepackage{graphicx}
\usepackage{color}

\newtheorem{theorem}{Theorem}
\newtheorem{lemma}[theorem]{Lemma}
\newtheorem*{lemma*}{Lemma}
\newtheorem{proposition}[theorem]{Proposition}
\newtheorem{corollary}[theorem]{Corollary}

\newtheorem{definition}[theorem]{Definition}

\newtheorem{remark}[theorem]{Remark}

\title[Multiplicative chaos measures for a random model of the Riemann 
zeta function]{Multiplicative chaos measures for a random model of the Riemann 
zeta function}

\date{\today}

\author[E. Saksman]{Eero Saksman}
\address{University of Helsinki, Department of Mathematics and Statistics,
         P.O. Box 68 , FIN-00014 University of Helsinki, Finland}
\email{eero.saksman@helsinki.fi}

\author[C. Webb]{Christian Webb}
\address{Department of mathematics and systems analysis, Aalto University, P.O.
Box 11000, 00076 Aalto, Finland}
\email{christian.webb@aalto.fi}


\thanks{The first author was supported by the Finnish Academy CoE in Analysis and Dynamics Research. The second author was supported by Academy of Finland.}

\newcommand{\Z}{\mathbb{Z}}
\newcommand{\R}{\mathbb{R}}
\newcommand{\C}{\mathbb{C}}
\newcommand{\E}{\mathbb{E}\, }

\newcommand{\Prob}{\mathbb{P}}

\newcommand{\add}{\; \; \underset{{\rm unif}}{\sim}\;\;}

\begin{document}

\begin{abstract}
We prove convergence of a stochastic approximation of powers of the  Riemann $\zeta$ function to a non-Gaussian multiplicative chaos measure, and prove that this measure is a non-trivial multifractal random measure. The results cover both the subcritical and critical chaos. A basic ingredient of the proof is a 'good' Gaussian approximation of the induced random fields  that is potentially of independent interest.
\end{abstract}

\maketitle

\section{Introduction}

The goal of this note is to study the multifractal behavior emerging from the complexity of the distribution of the prime numbers. Our approach is to consider a stochastic approximation to the Riemann $\zeta$ function and study its connection with random multifractal measures known as multiplicative chaos measures. This is strongly motivated by conjectures in \cite{FK} and recent results in \cite{ABH}, where this stochastic approximation to the $\zeta$ function was studied.

The main conjecture of \cite{FK} is that on a suitable scale, the logarithm of the $\zeta$-function on the critical line, far away from the origin, should look roughly like a log-correlated Gaussian field. For rigorous results in this direction, see \cite{Bourgade,BK}, and for further conjectures, see \cite{Ostrovsky}. Motivated by the conjectures in \cite{FK}, the stochastic approximation we consider has recently been studied in \cite{ABH}, where the authors proved that the maximum of this field behaves essentially as the maximum of a log-correlated Gaussian field (see e.g. \cite{DRZ,Madaule} for more on the maximum of log-correlated Gaussian fields). 

Log-correlated Gaussian fields are rough objects - they must be understood as random generalized functions, but as realized already by Kahane, some of their geometric properties can be studied by exponentiating these fields into random measures known as multiplicative chaos measures (see \cite{Kahane} for Kahane's original work, \cite{RV} for a recent review and \cite{Berestycki} for a concise proof of existence and uniqueness). Gaussian multiplicative chaos has recently found applications in two-dimensional quantum gravity \cite{DKRV,DS}, the study of random planar curves through conformal welding \cite{AJKS,Sheffield}, models for asset returns in mathematical finance \cite{BKM}, and random matrix theory \cite{FK,Webb}.

The major difference in our case is that the field is no-longer Gaussian and though some simple non-Gaussian cases have been studied \cite{BM} there is no general theory for studying such an object. Our main goal is to show that a corresponding object exists and it enjoys many of the properties Gaussian multiplicative chaos measures are known to have. 

Our approach is philosophically close to that of \cite{ABH}, but still slightly different. The main idea in their proof is to find a hidden tree structure which governs the main properties (such as the maximum) of the field. We on the other hand will exhibit a log-correlated Gaussian field, which provides a very good approximation of the field. While our calculations rely heavily on the model we study, this approach of Gaussian approximation might be useful for studying other "nearly Gaussian" multiplicative chaos measures.

The structure of this paper is the following:  in the next section we first introduce our model and state the main results. We then move on to proving that the non-Gaussian field can be approximated well by a log-correlated Gaussian one. We do this using a quantitative Gaussian approximation result for sums of random variables, whose proof we postpone to the last section. The approximation enables us to make use of standard Gaussian multiplicative chaos theory to prove that also the non-Gaussian chaos measure exists, both in the subcritical and critical cases. Making use of Gaussian multiplicative chaos theory, we additionally prove a result concerning the multifractal scaling of the non-Gaussian chaos measure. 

\section{The model and main results}

Let us begin by sketching how our model appears. Our discussion will be  imprecise and short. For more information, see \cite{ABH} and \cite[Section 3]{Bourgade}. 

The real question one is interested in is describing the behavior of the  $\zeta$ function on the critical line far away from the origin. This of course is an extremely difficult question so it's natural to try to simplify things. Assuming an Euler product representation for $\zeta$ on the critical line, one would formally have 

\begin{equation}
\notag \log \zeta(it+1/2)=-\sum_{p}\log(1-p^{-it-1/2})=\sum_p \sum_{k=1}^{\infty}\frac{1}{k}p^{-ikt -k/2},
\end{equation}

\noindent where the $p$-sum is over prime numbers. As one has quite fast decay in the summation variable $k$, it is natural to expect that the $k=1$ term would be the dominant part of this sum.  Thus one is lead to looking at the object

\begin{equation}
\notag \sum_{p}\frac{1}{\sqrt{p}}p^{-it}.
\end{equation}

Studying this is still too difficult, so one introduces randomness. We want to consider the behavior of the above object in the vicinity of a generic large point on the critical line. To formalize this, let $u$ be a random variable which is uniformly distributed on $[1,2]$ and let $T$ be a large parameter. Moreover, let $x\in[0,1]$ - so our generic large point is $iuT+1/2$ and we look at points near this, i.e. points corresponding to $t=uT+x$. Consider then the object 

\begin{equation}\label{eq:zetasum}
\sum_p \frac{1}{\sqrt{p}}p^{-ix}p^{-iu T}.
\end{equation}

If one considers only a fixed number of primes, say $p\leq N$ for some $N\in \Z_+$, then as $T\to\infty$,

\begin{equation}
\notag \left(p^{-iuT}\right)_{p\leq N}\stackrel{d}{\to}\left(e^{i\theta_p}\right)_{p\leq N},
\end{equation}
\noindent where $(\theta_p)_{p\leq N}$ are i.i.d. uniformly distributed on $[0,2\pi]$. This follows by observing that if $p_1,\ldots,p_K$ are distinct primes and $r_k$:s are integers, not all equal to zero, then in the limit  all the non-trivial joint moments vanish:
$$
\lim_{T\to\infty}\E \prod_{k=1}^K\big(p_k^{\pm iu T}\big)^{r_k}=\lim_{T\to\infty}\int_1^2\exp \Big(iTu\big(\sum_{k=1}^Kr_k\log p_k\big)\Big)du=0,
$$ 
as $\log(p_k)$:s are independent over rationals.
Now as we're interested in the size of $\zeta$, it's natural to only consider the real part of the logarithm and make the following definition. In order to facilitate definitions later on, we let $p_j$ stand for the $j$:th prime and define

\begin{definition}\label{def:xn}{\rm
Let $(\theta_p)_p$ be i.i.d. random variables that are uniformly distributed on $[0,2\pi]$ and indexed by prime numbers. We denote their law by $\Prob$ and integration with respect to this measure by $\E$. For $N\in \Z_+$ and $x\in[0,1]$ set

\begin{equation}\label{eq:xn}
\notag X_N(x)=\sum_{j=1}^N\frac{1}{\sqrt{p_j}}\left(\cos(x\log p_j)\cos \theta_{p_j}+\sin(x\log p_j)\sin\theta_{p_j}\right).
\end{equation}}
\end{definition}

\begin{remark}
{\rm One can check that as $N\to\infty$, the sequence of functions $(X_N)$ converges almost surely in some suitable Sobolev space of distributions to a non-trivial limit, say $X$, which is an honest random generalised function. A natural question arises, in analogy with random matrix theory (see e.g. \cite{HKO}),  whether the quantity \eqref{eq:zetasum} also  converges to the same limit as $T\to\infty$. However, one easily checks that this quantity does not converge locally in any reasonable Sobolev space of distributions for any fixed $T$. A  more natural way to make a more rigorous link to the $\zeta$-function would  be to study the convergence of suitable (smoothed) cuts of the series that depend on $T$.}
\end{remark}

This object came about when considering the logarithm of the $\zeta$ function, so it's natural to want to exponentiate it. It turns out that the correct way to understand this exponential is to view it as a positive measure. To get a better understanding of the measure, is is customary to add a further parameter that will enable studying the (random) $L^p$ norm of the "density" of the measure (though the limiting measure actually is almost surely not absolutely continuous with respect to the Lebesgue measure). We also need to normalize the measure suitably to obtain a non-trivial limiting object - our choice of normalization is such that the expectation of the total mass of the measure is equal to one - this and the independence of the summands ensures that  the sequences of measures forms a measure-valued martingale, which allows the use of standard limit theorems in order to define the limiting object.

\begin{definition}\label{def:mun}{\rm
For $\beta>0$, we consider the measure 

\begin{equation}
\notag \mu_{\beta,N}(dx)=\frac{e^{\beta X_N(x)}}{\E e^{\beta X_N(x)}}dx
\end{equation}

\noindent on $[0,1]$.}

\end{definition}

By the theory of martingales, the existence of a ${\rm weak}^*$-limit of the sequence
 $\mu_{\beta,N}(dx)$ is easy. However, that the limit is no-trivial is a more delicate issue, and our first main result guarantees that this is the case:

\begin{theorem}\label{th:main}
For $\beta\in(0,\beta_c)$, where $\beta_c=2$, the measure $\mu_{\beta,N}(dx)$ converges almost surely with respect to the weak topology of measures to a non-trivial random measure $\mu_\beta(dx)$. Actually, there is a Gaussian multiplicative chaos-measure $\nu_\beta$ on (0,1) such that $\mu_\beta=f\nu_\beta$, where the random multiplier  function $f$ is almost surely continuous and bounded form above and $\| f\|_{L^\infty (0,1)}$ possesses all moments.  One has  $\E \mu_\beta(0,1)^p<\infty$ for $p<4/\beta^2.$ For $\beta\geq \beta_c$, $\mu_{\beta,N}(dx)$ converges almost surely to the zero measure (with respect to the weak topology of measures).
\end{theorem}

Our second main result concerns the more difficult case of $\beta=\beta_c$, i.e. the critical case. Here convergence  to a non-trivial object is obtained if one  normalizes the measure in a slightly different  way:

\begin{theorem}\label{th:critical}
As $N\to\infty$, the measure $\sqrt{\log\log N}\mu_{\beta_C,N}(dx)$ converges in distribution (with respect to the weak topology) to a non-trivial random measure which is also absolutely continuous with respect to a Gaussian multiplicative chaos measure. Moreover, $\E \mu_{\beta_C}(0,1)^p<\infty$ for $p\in (0,1).$
\end{theorem}

These two results parallel very closely the type of behavior one has for Gaussian multiplicative chaos measures. 
Actually, they even imply   that in a suitable 'mesoscopic' scaling the approximating measures converge to an actual Gaussian multiplicative chaos measure, modified just by a scalar random multiplicative factor, see Remark \ref{rem:mesoscopic} below. 

There are a couple of issues  that one would expect to hold from the close relationship to the Gaussian case, but we do not touch on them in the present note. First of all we expect that in Theorem \ref{th:critical}, the convergence is not just in distribution, but also actually in probability -- in fact, this probably follows simply by slightly modifying some of the results in \cite{JS}. Moreover, it seems possible that applying our Gaussian approximation result one could obtain for $\beta>\beta_c$ another deterministic normalization under which the measures would converge (this time only in distribution) to a non-trivial limiting object and this this limiting object is a purely atomic measure. This is known as a freezing transition in the framework of physics of disordered systems, and is believed to be a universal phenomenon - see \cite{CLD}. 
Moreover, we suspect that it might be possible to prove that under a suitable deterministic shift, $\max_x X_N(x)$ converges in law to a non-trivial random object, whose distribution can be represented in terms of the critical measure as for log-correlated Gaussian fields. Our approach of expressing $X_N$ in terms of a Gaussian field means that the difficulty in proving all of these claims is in proving the corresponding result for the Gaussian field. While such results are known for some approximations of log-correlated Gaussian fields, the current knowledge is not sufficient to cover our case. For more information about these statements, see e.g. \cite{RV,DRSV1,DRSV2,MRV,DRZ,Madaule}.

A fundamental property of Gaussian multiplicative chaos measures is multifractality - or that the measure can't be described simply with a single scaling dimension, but needs a whole spectrum of them. There are different ways to make precise sense of this (in particular, in the theory of Gaussian multiplicative chaos, there are results about the so called KPZ-scaling of the dimension of the measure - see \cite{DS,RV2}), but we present the following simple result describing the non-trivial scaling of the (subcritical) measure.

\begin{proposition}\label{prop:multif}
Let $\beta<\beta_c$. Then there exists a $q_c=q_c(\beta)>1$ such that for $q\in(0,q_c)$ and any $x\in (0,1)$

\begin{equation*}
\lim_{r\to 0}\frac{\log\E(\mu_\beta(B(x,r))^{q})}{\log r}=\left(1+\frac{\beta^{2}}{2}\right)q-\frac{\beta^{2}}{2}q^{2}.
\end{equation*}
\end{proposition}

Let us finally briefly outline our approach to proving convergence of $\mu_{N,\beta}$. As in Kahane's original theory, $\mu_{\beta,N}$ is a measure valued martingale - in particular, for each continuous function $f:[0,1]\to \R$, $\mu_{\beta,N}(f)$ is a martingale. So to prove convergence to a non-trivial object, for $\beta<\beta_c$ it is enough for us to demonstrate that this martingale is uniformly integrable. As in the Gaussian case, we'll prove that for $\beta<\beta_c$, the martingale is bounded in $L^p$ for some $p>1$. This will in fact follow from the our representation of $\mu_{\beta,N}$ being absolutely continuous with respect to an approximation of a Gaussian chaos measure and the Radon-Nikodym derivative being very well behaved. For the  case $\beta=\beta_c$ we need to be much more careful in  choosing the approximative Gaussian field, but after that the result is obtained by  applying uniqueness results for critical Gaussian chaos contained in \cite{JS}.

The  Gaussian approximation we need is contained in the following:

\begin{theorem}\label{th: gaussian_appro}  For each $N\geq 1$ there is the decomposition
\begin{equation}
\notag X_N(x)=G_N(x)+E_N(x),
\end{equation}
where $G_N(x)$ is a Gaussian field that has the covariance structure of a standard smooth approximation to a log-correlated field $($see Lemma
\ref{le:covariance_approximation} below$)$, and $E_N$ is continuous and converges a.s. uniformly to continuous function $E$. Moreover, the maximal error in the approximation has finite exponential moments:
$$
\E \exp\big( \lambda \sup_{N\geq 1, x\in [0,1]}E_N(x)\big) \; <\; \infty \quad {\rm for\;  all} \;\; \lambda >0.
$$
\end{theorem}
The idea behind the Gaussian approximation is simply that in definition \eqref{eq:xn} we may divide the sum into suitable blocks and use the slowly varying nature of $p\to\log p$ to 'freeze' the $x$-dependence inside each block, and obtain a Gaussian approximation by a simple coupling argument. The basic input from number theory needed is the prime number theorem with a good bound for the error term.

In what follows, for the sake of non-initiated reader we have not striven for a condensed exposition but instead attempt to provide full details  even for the somewhat repetitive  parts of the argument.

\section{A Gaussian approximation for the field}
The goal of this section is to prove that we can indeed write $X_N(x)=G_N(x)+E_N(x)$, where $G_N$ converges to a log-correlated Gaussian field, and $E_N$ converges to a continuous function. This will be carried out in steps. First we'll prove things along a suitable subsequence of $N$s making use of a Gaussian approximation theorem for sums of independent random variables, and later extend the result to all $N$.

 As mentioned earlier, we'll want to split the field into a sum over blocks, where within the blocks, the quantities $\log p$ are roughly constant, and perform a Gaussian approximation on each block separately. To make this formal, let $\mathbf{P}=\lbrace p_1,p_2,..\rbrace$ be the set of primes (indexed in increasing order) and let $(r_k)_{k=1}^{\infty}$ be a sequence of strictly increasing  positive integers with $r_1=1$. The idea is that $\lbrace p_{r_k},...,p_{r_{k+1}-1}\rbrace$ will be the set of primes appearing in the block we've mentioned. 

Later on, we'll discuss what we precisely require of the sequence $r_k$, but for now we note that if we want some kind of central limit theorem to take effect within a block we need $r_{k+1}-r_k\to \infty$ as $k\to\infty$. On the other hand, to have $\log p_{r_k}\approx \log p_{r_{k+1}}$, we'll want (by the prime number theorem) $r_{k+1}/r_k\to 1$ as $k\to\infty$. To apply our Gaussian approximation result - Proposition \ref{le2} - without much further calculations, we also assume that $p_{r_{m+1}-1}/p_{r_m}\leq 2$. Let us further assume that $r_{m+1}-r_m>1$ for all $m$.

We then define the "blocks" of the field as well as our "freezing approximation".

\begin{definition}\label{def:blocks}{\rm
For $(r_m)_{m=1}^\infty$ as above, define for $x\in[0,1]$ and $m\geq 1$:

\begin{equation}\label{eq:ym}
Y_m(x)=\sum_{k=r_m}^{r_{m+1}-1}\frac{1}{\sqrt{p_k}}\left(\cos(x\log p_k)\cos \theta_{p_k}+\sin(x\log p_k)\sin\theta_{p_k}\right).
\end{equation}

Consider also the approximation to this where the $x$-dependent terms within each block are frozen:

\begin{align}\label{eq:ymtil}
\widetilde{Y}_m(x)&=\cos(x\log p_{r_m})\sum_{k=r_m}^{r_{m+1}-1}\frac{1}{\sqrt{p_k}}\cos \theta_{p_k}+\sin(x\log p_{r_m})\sum_{k=r_m}^{r_{m+1}-1}\frac{1}{\sqrt{p_k}}\sin \theta_{p_k}\\
\notag &=:\cos(x\log p_{r_m}) C_m+\sin(x\log p_{r_m})S_m.
\end{align}}

\end{definition}

The Gaussian approximation needed will be based on the following result. We state it in a slightly more general form than we actually need here, since this turns out to be useful in further study of non-Gaussian chaos models \cite{J}.

\begin{proposition}\label{le2}  \quad{\bf (i)}\quad Assume that $d\geq 2$ and $H_j=(H_j^{(1)},\ldots H_j^{(d)})$, $j\in\{1,\ldots, n\}$ are independent and symmetric $\R^d$-valued random variables with $b_0^{-1}\leq c_j:=d^{-1}{\rm Tr\,}({\rm Cov\,}(H_j))\leq b_0$ for all   $j\in\{1,\ldots, n\}$, where $b_0>0$.    Assume also that the following uniform exponential bound holds for some $b_1,b_2>0$:
\begin{equation}\label{eq:exp}
\E \exp(b_1|H_k|)\leq b_2\quad\textrm{for all}\quad k=1,\ldots , n.
\end{equation}
Then there is a $d$-dimensional  Gaussian random variable $U$  with 
$$
{\rm Cov\,}(U)=\big(\sum_{j=1}^n c_j\big)^{-1}\big(\sum_{j=1}^n {\rm Cov\,}(H_j) \big),\qquad {\rm Tr \,} ({\rm Cov\,}(U))=d,
$$
and such that  the difference
$$
V:=U-\big(\sum_{j=1}^n c_j\big)^{-1/2}\big(\sum_{j=1}^n H_j \big)
$$
satisfies
\begin{equation}\label{e30}
\E |V|\leq a_1n^{-\beta}.
\end{equation}
Above $\beta =\beta(d)>0$ depends only on the dimension and
$a_1$ on $d, b_0,b_1,b_2$. Moreover, $U$ can be chosen to be measurable with respect to $\sigma (G, H_1,\ldots H_n),$
where $G$ is a d-dimensional standard Gaussian independent of the $H_j$:s. In addition, there is the exponential estimate
\begin{equation}\label{e32}
\E \exp(\lambda |V|)\leq 1+ a_2 e^{a_3 \lambda^2}n^{-a_4}\qquad\textrm{for}\quad 0\leq\lambda\leq a_5n^{1/2},
\end{equation} 
where the constants $a_2,a_3,a_4, a_5>0$ depend only on $b_0, b_1, b_2$ and the dimension $d.$ Here $a_2\in (0,1/2).$

In the case where the variables $H_k$ are uniformly bounded, say $|H_k|\leq b_3$ for all $k,$ then \eqref{e32}  holds true for all $\lambda >0$, where now the constants $a_2,a_3,a_4$ may   also depend on $b_3, $ and there are   constants $a_6,a_7,a_8>0$ that  depend only on $b_0, b_1,b_2,b_3,d$ so that
\begin{equation}\label{e31}
\E \exp(a_6|V|^2)\leq 1+a_7n^{-a_8},
\end{equation}

\smallskip

\noindent {\bf (ii)}\quad If we assume that $Cov(H_j)=c_jd^{-1}I$, where $I$ the the $d\times d$ identity matrix, and the dimension $d\geq 1$ is arbitrary, then the conclusion \eqref{e30} can be strengthened to 
\begin{equation}\label{e30'}
\E |V|\leq a_1\log(n)^{d+1}n^{-1/2}.
\end{equation}

\end{proposition}
\noindent We will postpone the proof of this result to a later section.

We'll now consider what kind of Gaussian approximation this implies in our case - our aim is to apply Proposition \ref{le2} to approximate $(C_m,S_m)$ by a $\R^{2}$-valued Gaussian random variable. To do this, we need to scale things a bit differently. Define the following sequence of $\R^{2}$-valued random variables (so in the setting of Proposition \ref{le2}, $d=2$)

\begin{equation}\label{eq:hj}
H_{j,m}=\left(\frac{\sqrt{p_{r_{m+1}-1}}}{\sqrt{p_{r_m-1+j}}}\cos \theta_{p_{r_m-1+j}},\frac{\sqrt{p_{r_{m+1}-1}}}{\sqrt{p_{r_m-1+j}}}\sin \theta_{p_{r_m-1+j}}\right).
\end{equation}

We now have 

\begin{equation*}
|H_{j,m}|^{2}\leq \frac{p_{r_{m+1}-1}}{p_{r_m}}\leq 2
\end{equation*}

\noindent and 

\begin{equation*}
Cov(H_{j,m})=\frac{1}{2}\frac{p_{r_{m+1}-1}}{p_{r_m-1+j}} I=c_{j,m}I
\end{equation*}

\noindent where $1/2\leq c_{j,m}\leq 1$. In this notation, we have 

\begin{equation*}
(C_m,S_m)=\frac{1}{\sqrt{p_{r_{m+1}-1}}}\sum_{j=1}^{r_{m+1}-r_m}H_{j,m}=\left({\frac{1}{2}\sum_{j=1}^{r_{m+1}-r_m}\frac{1}{p_{r_m-1+j}}}\right)^{1/2}\frac{1}{\sqrt{\sum_{j=1}^{r_{m+1}-r_m}c_{j,m}}}\sum_{j=1}^{r_{m+1}-r_m}H_{j,m}
\end{equation*}
Proposition \ref{le2} (ii) thus yields a sequence of independent standard two-dimensional normal variables $(V^{(1)}_n,V^{(2)}_n)$ for all $n\geq 1$, so that the distance between $(C_m,S_m)$ and 

\begin{equation*}
\left(\frac{1}{2}\sum_{j=1}^{r_{m+1}-r_m}\frac{1}{p_{r_m-1+j}}\right)^{1/2}(V_m^{(1)},V_m^{(2)})
\end{equation*}

\noindent can be controlled.


We may assume that our probability space is large enough for us to write for each $m\geq 1$ and $i\in\{1,2\}$
$$
\sqrt{\frac{1}{2}\sum_{j=1}^{r_{m+1}-r_m}\frac{1}{p_{r_m-1+j}}}V_m^{(i)}=\sum_{j=r_m}^{r_{m+1}-1} \frac{1}{\sqrt{2p_j}}W^{(i)}_j,
$$
where the $W^{(i)}_j$:s are independent standard normal random variables for all $j\geq 1$ and $i\in\{1,2\}$. Finally we can define our Gaussian approximation to the field, its blocks, and frozen versions of the blocks. 

\begin{definition}\label{de:gaussian}{\rm
Let $(W_k^{(j)})_{k\geq 1,j\in\lbrace 1,2\rbrace}$ be i.i.d. standard Gaussians. For any $N\geq1$ and $x\in[0,1]$ the Gaussian approximation of the field $X_N$ is given by the Gaussian field
\begin{equation}\label{eq:gn}
G_N(x):=\sum_{j=1}^N\frac{1}{\sqrt{2p_j}}\left(W^{(1)}_j\cos(x\log p_j)+W^{(2)}_j\sin(x\log p_j)\right).
\end{equation}

Moreover, we define the blocks of $G_N$ as 

\begin{equation}\label{eq:ygm}
Z_m(x)=\sum_{k=r_m}^{r_{m+1}-1}\frac{1}{\sqrt{2p_k}}\left(\cos(x\log p_k)W_k^{(1)}+\sin(x\log p_k)W_k^{(2)}\right)
\end{equation}

\noindent and a "frozen" version of the block as 

\begin{equation}
\widetilde Z_m(x)=\cos(x\log p_{r_m})b_m V_m^{(1)}+\sin(x\log p_{r_m})b_m V_m^{(2)},
\end{equation}

\noindent where

\begin{equation*}
b_m=\sqrt{\frac{1}{2}\sum_{j=1}^{r_{m+1}-r_m}\frac{1}{p_{r_m-1+j}}}.
\end{equation*}}

\end{definition}

We then the  start the analysis of the error produced by our approximations. This is first performed only for sums over full blocks. We introduce some notation for the errors. 
Let us call the error we make by approximating our "frozen" field by the "frozen" Gaussian one by 
\begin{equation}
\widetilde E_{1,n}(x):=\sum_{m=1}^{n}(\widetilde{Y}_m(x)-\widetilde Z_m(x)),\qquad x\in (0,1).
\end{equation}
In a similar vein, the error obtained from the "freezing procedure" is denoted by  
\begin{equation}
\widetilde E_{2,n}(x):=\sum_{m=1}^{n}\big(Y_m(x)-\widetilde{Y}_m(x)+\widetilde Z_m(x)-Z_m(x)\big),\qquad x\in (0,1).
\end{equation}
whence the total error  can be written as
\begin{equation}
\widetilde E_{n}(x):=\widetilde E_{1,n}(x)+ \widetilde E_{2,n}(x).
\end{equation}

We study first the error $\widetilde{E}_{1,n}$.

\begin{lemma}\label{le:E1}
Assume $($in addition to our previous constraints on $(r_m))$ that 
\begin{equation}
\sum_{m=1}^{\infty}(r_{m+1}-r_m)^{-a_4}<\infty,
\end{equation}
where $a_4$ is the constant from Proposition \ref{le2}.
Then, almost surely  there exists a continuous limit function
\begin{equation}\label{eq:171}
\widetilde E_1(x):=\lim_{n\to\infty} \widetilde E_{1,n}(x),
\end{equation}
where the convergence is in the sup-norm over $(0,1)$. Moreover, one has 
\begin{equation}\label{eq:172}
\E\exp(\lambda \sup_{0\leq n'<n}||\widetilde E_{1,n}-\widetilde E_{1,n'}||_{L^{\infty}(0,1)})<\infty\quad \textrm{for all}\quad \lambda >0,
\end{equation}
where one applies the convention $\widetilde E_{1,0}(x) \equiv 0$. In particular,
\begin{equation}\label{eq:173}
\E\exp(\lambda ||\widetilde E_1||_{L^{\infty}(0,1)})<\infty\quad \textrm{for all}\quad \lambda >0,
\end{equation}
\end{lemma}
\begin{proof}
To prove convergence of $\widetilde{E}_{1,m}$, we note that 

\begin{align*}
||\widetilde{E}_{1,m-1}-\widetilde{E}_{1,m}||_{L^\infty(0,1)}&\leq |(C_m,S_m)-b_m(V_m^{(1)},V_m^{(2)})|\\
\notag &=b_m\left|\frac{1}{\sqrt{\sum_{j=1}^{r_{m+1}-r_m}c_{j,m}}}\sum_{j=1}^{r_{m+1}-r_m}H_{j,m}-(V_m^{(1)},V_m^{(2)})\right|.
\end{align*}

We then recall that we assumed that $r_{m+1}/r_m\to 1$ so we see from the prime number theorem (and a crude estimate on the sum) that for some constant $C>0$

\begin{equation*}
b_m^2\leq C \frac{r_{m+1}-r_m}{r_m\max(\log r_m,1)}
\end{equation*}

\noindent so we see that $b_m\to 0$ and in particular, it is bounded. Thus by Proposition \ref{le2} \eqref{e32}, we have for some constants $C,\widetilde{C}$

\begin{align*}
\E||\widetilde{E}_{1,m-1}-\widetilde{E}_{1,m}||_{L^\infty(0,1)}&\leq C\E( e^{|(\sum_{j=1}^{r_{m+1}-r_m}c_{j,m})^{-1/2}\sum_{j=1}^{r_{m+1}-r_m}H_{j,m}-(V_m^{(1)},V_m^{(2)})|}-1)\\
&\leq \widetilde{C} (r_{m+1}-r_m)^{-a_4}.
\end{align*}

Thus by our assumption on $(r_m)$, the series

\begin{equation}
\widetilde{E}_1=\sum_{m=1}^\infty (\widetilde{E}_{1,m}-\widetilde{E}_{1,m-1})
\end{equation}

\noindent converges in $L^\infty(0,1)$.

We next use the crude estimate
\begin{align*}
\sup_{0\leq n'<n} ||\widetilde E_{1,n}-\widetilde E_{1,n'}||_{L^{\infty}(0,1)}&\leq \sum_{m=1}^{\infty}(|b_m V_m^{(1)}-C_m|+|b_m V_m^{(2)}-S_m|)\\
&\leq \sqrt{2}\sum_{m=1}^{\infty}|(C_m,S_m)-b_m(V_m^{(1)},V_m^{(2)})|,
\end{align*}
so that by independence and Proposition \ref{le2}
\begin{align*}
\E\exp(\lambda \sup_{0\leq n'<n}||\widetilde E_{1,n}-\widetilde E_{1,n'}||_{L^{\infty}(0,1)})&\leq \prod_{m=1}^{\infty}\E e^{\sqrt{2}\lambda |(C_m,S_m)-b_m(V_m^{(1)},V_m^{(2)})|}\\
\notag &\leq \prod_{m=1}^{\infty}\left(1+a_2 e^{2a_3\lambda^{2}b_m^{2}}(r_{m+1}-r_m)^{-a_4}\right)
\end{align*}

As we saw that $b_m$ is bounded and we find for some constant $C$ (depending on $\lambda$) that 

\begin{equation*}
\E\exp(\lambda \sup_{0\leq n'<n}||\widetilde E_{1,n}-\widetilde E_{1,n'}||_{L^{\infty}(0,1)})\leq \prod_{m=1}^{\infty}\left(1+C(r_{m+1}-r_m)^{-a_4}\right)\leq e^{C\sum_{m=1}^{\infty}(r_{m+1}-r_m)^{-a_4}},
\end{equation*}
and  \eqref{eq:172} follows. Finally, \eqref{eq:173} is an obvious  consequence of \eqref{eq:172}.

\end{proof}

Let us then estimate the error due to the  freezing procedure.

\begin{lemma}\label{le:E2}
Assume that the sequence $(r_m)$ is chosen so that 
\begin{equation}\label{eq:rcond}
\sum_{m=1}^{\infty}\frac{(r_{m+1}-r_m)(p_{r_{m+1}}-p_{r_m})^{2}}{r_m^{3}}<\infty.
\end{equation}
Then, almost surely  there exists the continuous limit function
\begin{equation}\label{eq:181}
\widetilde E_2(x):=\lim_{n\to\infty} \widetilde E_{2,n}(x),
\end{equation}
where the convergence is in the sup-norm over $(0,1)$.  Moreover,  for small enough $a>0$ we have both
\begin{equation}\label{eq:182}
\E \exp\big(a \|\widetilde E_{2}\|_{L^{\infty}(0,1)}^{2}\big)<\infty
\end{equation}
and 
\begin{equation}\label{eq:183}
\E \exp\big(a \sup_{0\leq n'<n}\|\widetilde E_{2,n}-\widetilde E_{2,n'}\|_{L^{\infty}(0,1)}^{2}\big)<\infty
\end{equation}
Consequently, 
\begin{equation}\label{eq:184}
\E \exp\big(\lambda \|\widetilde E_{2}\|_{L^{\infty}(0,1)}\big)<\infty \quad  \textrm{for all}\quad \lambda >0\qquad \textrm{and}
\end{equation}
\begin{equation}\label{eq:185}
\E \exp\big(\lambda \sup_{0\leq n'<n}||\widetilde E_{2,n}-\widetilde E_{2,n'}||_{L^{\infty}(0,1)}\big)<\infty \quad  \textrm{for all}\quad \lambda >0.
\end{equation}
\end{lemma}
\begin{proof}
For the reader's convenience we first recall a standard estimate for the sup-norm of a given function $g\in C^1(0,1)$. For arbitrary $x,y\in(0,1)$ we may estimate 
\begin{equation*}
|g(x)|=|g(y)+\int_y^{x}g'(t)dt|\leq |g(y)|+\int_0^1|g'(t)|dt.
\end{equation*}
Since $x$ is arbitrary, by integrating with respect to $y$ and using the  Cauchy-Schwarz inequality we obtain
\begin{equation}\label{eq:sobolev}
\|g\|_{L^\infty (0,1)}\leq \int_0^1(|g'(t)|+|g(t)|)dt\leq  2 \Big(\int_0^1(|g'(t)|^2+|g(t)|^2)dt\Big)^{1/2}:=2||g||_{W^{1,2}(0,1)}
\end{equation}

Recall that 
\begin{align}\label{eq:190}
Y_m(x)- \widetilde{Y}_m(x)&+\widetilde{Z}_m(x)-Z_m(x)\\
&=\sum_{k=r_m}^{r_{m+1}-1}\frac{1}{\sqrt{p_k}}\big(\cos \theta_{p_k}-2^{-1/2}W^{(1)}_k\big)\big(\cos(x\log p_k))-\cos(x\log p_{r_m})\big)\notag \\
&\qquad + \sum_{k=r_m}^{r_{m+1}-1}\frac{1}{\sqrt{p_k}}\big(\sin \theta_{p_k}-2^{-1/2}W^{(2)}_k\big)\big(\sin(x\log p_k))-\sin(x\log p_{r_m})\big).\nonumber
\end{align}
As $(Y_m-\widetilde{Y}_m+\widetilde{Z}_m-Z_m)$ is independent of $(Y_{m'}-\widetilde{Y}_{m'}+\widetilde{Z}_{m'}-Z_{m'})$ for $m\neq m'$, and these objects have vanishing expectation, we have
\begin{align}
\E\|\widetilde E_{n,2}-\widetilde E_{n',2}||_{L^{\infty }(0,1)}^{2}&\leq 2\sum_{m=n'+1}^{n}\int_0^{1}\Big(\E(Y_m(x)-\widetilde{Y}_m(x)+\widetilde{Z}_m(x)-Z_m(x))^{2}\\
\notag & \qquad +\E(Y'_m(x)-\widetilde{Y}'_m(x)+\widetilde{Z}_{m}'(z)-Z_m'(x))^{2}\Big)dx.
\end{align}
Observe that  for all $x\in (0,1)$ we have
\begin{eqnarray}\label{eq:191}
&&\big(\cos(x\log p_k)-\cos(x\log p_{r_m})\big)^{2}+\big(\sin(x\log p_k)-\sin(x\log p_{r_m})\big)^{2}\\&=&2( 1-\cos(x (\log p_k-\log p_{r_m})))
\;\; \leq \;\;(\log (p_k/p_{r_m}))^2.\nonumber
\end{eqnarray}
Moreover,
\begin{align*}
|\log p\sin(x\log p)-\log q\sin(x\log q)|&\leq \log p|\sin(x\log p)-\sin(x\log q)|+|\log p-\log q|\\
&\leq 2\log p |\log p-\log q|
\end{align*}
and as a similar estimate is valid for the cosine term, we deduce that 
\begin{eqnarray}\label{eq:192}
&&\Big(\frac{d}{dx}\big(\cos(x\log p_k)-\cos(x\log p_{r_m})\big)\Big)^{2}+\Big(\frac{d}{dx}\big(\sin(x\log p_k)-\sin(x\log p_{r_m})\big)\Big)^{2}\\
&\leq & 
8\big(\log(p_{r_m})\log (p_k/p_{r_m})\big)^2.\nonumber
\end{eqnarray}
By dividing the sum \eqref{eq:190} into two parts\footnote{Here one should note that the variable $\theta_{p_j}$ inside a block is not necessarily independent of any of the variables $W^{(i)}_j$ inside the same block!}, where the first one corresponds to the random variables $\theta_{p_j}$ and the second one the random variables $W^{(i)}_j$ (and then use the elementary inequality $(a+b)^2\leq 2(a^2+b^2)$), we may perform for both parts an identical computation that uses independence and the previous estimates
to obtain for any $x\in(0,1)$
\begin{align}\label{eq:193}
\E(Y_m(x)-\widetilde{Y}_m(x)&+\widetilde{Z}_m(x)-Z_m(x))^{2} +\E(Y'_m(x)-\widetilde{Y}'_m(x)+\widetilde{Z}_{m}'(z)-Z_m'(x))^{2}\\
& \leq 36\sum_{k=r_m}^{r_{m+1}-1}\frac{1}{p_k}\big(\log(p_{r_m})\log (p_k/p_{r_m})\big)^2\notag\\
&\lesssim (r_{m+1}-r_m)\frac{\log^{2}p_{r_m}}{p_{r_m}^{3}}(p_{r_{m+1}-1}-p_{r_m})^{2}.\nonumber
\end{align}
By summing over $m\in\{ n'+1,\ldots ,n\}$ and integrating over $(0,1)$ it follows that
\begin{equation}\label{eq:194}
\E \| \widetilde E_{n,2}-\widetilde{E}_{n',2}\|_{L^{\infty }(0,1)}^{2}\lesssim\sum_{m=n'+1}^n (r_{m+1}-r_m)\frac{\log^{2}p_{r_m}}{p_{r_m}^{3}}(p_{r_{m+1}-1}-p_{r_m})^{2}.
\nonumber
\end{equation}
Then Levy's inequality (see \cite[Lemma 1., p. 14 ]{Kahane2}, applied  here to our  $C(0,1)-$valued symmetric random variables) yields that

\begin{equation}\label{eq:195}
\E \Big(\sup_{n'\leq r\leq n}\| \widetilde E_{r,2}-\widetilde E_{n',2}\|_{L^{\infty }(0,1)}^{2}\Big)\lesssim\sum_{m=n'+1}^n (r_{m+1}-r_m)\frac{\log^{2}p_{r_m}}{p_{r_m}^{3}}(p_{r_{m+1}-1}-p_{r_m})^{2}.
\end{equation}
Using the prime number theorem, we can bound this series by one appearing in the assumptions of this lemma. Thus the series above converges and this enables us to pick a subsequence $(n_\ell)$  with the property
$$
\sum_{m=n_\ell+1}^{n_{\ell+1}} (r_{m+1}-r_m)\frac{\log^{2}p_{r_m}}{p_{r_m}^{3}}(p_{r_{m+1}-1}-p_{r_m})^{2}<\ell^{-6}\quad\textrm{for all}\quad \ell\geq 1.
$$ 
 Borel-Cantelli lemma combined with \eqref{eq:195} yields an almost surely finite   index $\ell_0(\omega)$ such that
$$
\sup_{n_\ell+1\leq u\leq n_{\ell+1}} \| \widetilde E_{u,2}-\widetilde E_{n_\ell,2}\|_{L^{\infty }}\leq\ell^{-2}\quad \textrm{for}\quad \ell\geq \ell_0(\omega),
$$
summing over $l$, this yields the statement \eqref{eq:181} on the convergence.

\smallskip

In order to consider the double exponential integrability of our random variable, let us define the sequence $(c_k)$ by setting 
$c_k=8p_{r_m}^{-1/2}\log(p_{r_m})\log (p_k/p_{r_m})$ for $r_m\leq k\leq r_{m+1}-1$ and $m\geq 1.$ Fix any $x\in (0,1)$ and observe that \eqref{eq:190} and our estimates \eqref{eq:191} and
\eqref{eq:192} show that we may write
$$
\widetilde  E'_{2}(x) =\sum_{k=1}^\infty A_k(x)\qquad\textrm{and}\qquad \widetilde  E_{2}(x) =\sum_{k=1}^\infty B_k(x),
$$
where the symmetric  random variables $A_k(x)$  
can be written in the form $A_k(x)= A_{1,k}(x)+ A_{2,k}(x)$, so  that the random variables $A_{1,k}(x)$ in turn are independent and satisfy the bound $|A_{1,k}(x)|\leq c_k$ 
for all $k$.  In turn, the variables $A_{2,k}$ are independent centered Gaussians with $\E (A_{2,k}(x))^2\leq c_k^2$. Note that in particular, the argument for uniform convergence of $\widetilde{E}_{n,2}$ goes through essentially unchanged for proving uniform convergence of $\widetilde{E}_{n,2}'$ so we can indeed differentiate term by term. Our previous computations for \eqref{eq:193} verify that $\sum_{k=1}^\infty c_k^2<\infty.$ A similar decomposition is valid for the terms $B_k(x)$ with the same bounds. Azuma's inequality applied to the bounded summands, and a trivial estimate to the Gaussian sums (along with H\"older to allow us to consider the Gaussian and non-Gaussian case separately) yields  for small enough $a>0$ the existence of a finite constant $C$  such that both
\begin{equation}\label{e20}
\E \exp\big(a  |\widetilde E_2(0)|^2\big)\leq C\qquad \textrm{ and}\qquad   \E \exp\big(a  |\widetilde E'_2(x)|^2\big)\leq C\quad \textrm{ for all}\quad  x\in [0,1].
\end{equation}
 In particular,  Fubini  yields that
 \begin{equation}\label{e21}
 \E \int_0^1  \exp\big(a  |\widetilde E'_2(x)|^2\big)dx<\infty.
 \end{equation}
 By the first inequality in \eqref{e20} it is enough to show that $\sup_{x\in(0,1)}|E_2(x)-E_2(0)|$ has the desired exponential integrability. However,
 now  $\sup_{x\in(0,1)}|\widetilde E_2(x)-\widetilde E_2(0)|\leq \int_0^1|\widetilde E_2'(x)|dx$ and since $t\mapsto \exp(at^2)$ is convex we obtain by Jensen's inequality
 \begin{equation}\label{e22}
\exp\big(a \Big( \int_0^1|\widetilde E_2'(x)|dx\Big)^2\big)\leq\int_0^1  \exp\big(a  |\widetilde E'_2(x)|^2\big)dx
  \end{equation}
 and \eqref{eq:182} is obtained by taking expectations and remembering \eqref{e21}. This improves to  \eqref{eq:183} by Levy's inequality, perhaps by making $a$ smaller if needed, and  finally  \eqref{eq:184} and \eqref{eq:185} follow immediately.

\end{proof}

We next combine the error estimates proven so far and make the final choice for the subsequence $(r_m)$. For that purpose we need the following well-known lemma, whose proof we include for the reader's convenience.

\begin{lemma}\label{le:li-inverse}  For large enough $n$ it holds that
\begin{equation}\label{eq:prime}
-ne^{-\sqrt{\log n}}\lesssim p_n- \mathrm{Li}^{-1}(n)\lesssim ne^{-\sqrt{\log n}}.\nonumber
\end{equation}
\end{lemma}
\begin{proof} We note first that the inverse $ \mathrm{Li}^{-1}$ is convex since $ \mathrm{Li}$ itself is concave. Furthermore,
we have $( \mathrm{Li}^{-1})'(x)=\log( \mathrm{Li}^{-1}(x))\leq \log(2x\log(x))\leq 2\log(x)$ for large enough $x.$ Hence, as a suitable  quantitative version of the prime number theorem verifies that for any $c\geq 1$
$|\pi(x)- \mathrm{Li}(x)|=O\big(x\exp(-c\sqrt{log x})\big)$, so we have $n=\pi(p_n)\leq  \mathrm{Li}(p_n)+ne^{-2\sqrt{\log n}}$. In particular, 
$$
p_n\geq \mathrm{Li}^{-1}(n- ne^{-2\sqrt{\log n}})\geq  \mathrm{Li}^{-1}(n)-ne^{-2\sqrt{\log n}}( \mathrm{Li}^{-1})'(n)\geq \mathrm{Li}^{-1}(n)-ne^{-\sqrt{\log n}}.
$$
The proof of the other direction is analogous.
\end{proof}

\begin{proposition}\label{pr2} Fix $\alpha\in (0,2/5)$ and define $r_m=\lfloor \exp (m^\alpha)\rfloor$. Then  the combined error $\widetilde E_{n}(x)=\widetilde{E}_{n,1}(x)+
\widetilde  E_{n,2}(x)$ a.s. converges uniformly on $(0,1)$ to a continuous limit function
$$
E(x):=\lim_{n\to\infty} (\widetilde E_{n,1}(x)+
\widetilde  E_{n,2}(x)).
$$
Moreover, it holds that
\begin{equation}\label{eq:204}
\E \exp\big(\lambda \|  E\|_{L^{\infty}(0,1)}\big)<\infty \quad \textrm{and}\quad \E \exp\big(\lambda \sup_{0\leq n'<n}\| \widetilde E_{n}-\widetilde E_{n'}\|_{L^{\infty}(0,1)}\big)<\infty\quad  \textrm{for all}\quad \lambda >0.
\end{equation}
 \end{proposition}
\begin{proof}
We first recall the condition of Lemma \ref{le:E1} - namely that the first error term converges as soon as
 \begin{equation}\label{e15}
\sum_{m=1}^\infty \big(r_{m+1}-r_m\big)^{-a_4} <\infty.
\end{equation}
Lemma \ref{le:li-inverse}  yields for our sequences that $p_{r_{m+1}}-p_{r_m}\lesssim (r_{m+1}-r_m)\log r_m + r_me^{-\sqrt{\log r_m}}.$ By plugging this into condition
\eqref{eq:rcond} we see that a sufficient condition to apply Lemma \ref{le:E2} in order to control the second error term is given by the pair of conditions
\begin{equation}\label{eq:205}
 \sum_{m=1}^\infty\bigg(\frac{r_{m+1}-r_m}{r_m}\bigg)^3\log^2 (r_m)<\infty\quad\textrm{and}\quad
  \sum_{m=1}^\infty e^{-3\sqrt{\log r_m}}<\infty.
\end{equation} 
Finally, it remains to observe that the choice $r_m=\lfloor \exp (m^\alpha)\rfloor$ satisfies both \eqref{e15} and \eqref{eq:205} as soon as $\alpha\in (0,2/5).$
\end{proof}

To complete the approximation procedure, we verify that the fields $G_N$ are good approximations also for indices
$N$ inside the interval $r_m\leq N<r_{m+1}$.
\begin{theorem}\label{th:g-approx} Denote the total error of the Gaussian approximation by setting
\begin{equation}\label{eq:301}
E_N(x):=X_N(x)-G_N(x)\quad \textrm{for}\quad N\geq 1\quad\textrm{and}\quad x\in(0,1).
\end{equation}
Then, almost surely,  $E_{N}(x)$ converges uniformly on $(0,1)$ to a continuous limit function
$$
E(x):=\lim_{N\to\infty} E_{N}(x),
$$
where the obtained limit is of course the same as in Proposition{\rm  \ref{pr2}}.
Moreover, it holds that
\begin{equation}\label{eq:204v}
\E \exp\big(\lambda \|  E\|_{L^{\infty}(0,1)}\big)<\infty \quad \textrm{and}\quad \E \exp\big(\lambda \sup_{N\geq 1}\|E_{N}\|_{L^{\infty}(0,1)}\big)<\infty\quad  \textrm{for all}\quad \lambda >0.
\end{equation}
\end{theorem}
\begin{proof}
After proposition \ref{pr2} it is enough to show that any given partial sum of the original series is in fact well approximated by the sum of the blocks below it, and that a similar statement holds also true for the Gaussian approximation series.  Let us fix $m\geq 1$ and recall our notation
$$
Y_m(x)=\sum_{k=r_m}^{r_{m+1}-1}\frac{1}{\sqrt{p_k}}\Big(\cos(\theta_{p_k})\cos \big(\log p_k x\big)+\sin(\theta_{p_k})\sin \big(\log p_k x\big)\Big)
\;=:\;\sum_{k=r_m}^{r_{m+1}-1}A_k(x), 
$$
which is just the partial sum of our original field $X_N$ corresponding to the $m$:th block.  Observing first that
\begin{eqnarray*}
\sum_{k=r_m}^{r_{m+1}-1}\frac{\log^2 p_k}{p_k}&\lesssim& \sum_{k=\lfloor e^{m^\alpha}\rfloor}^{\lfloor e^{(m+1)^\alpha}\rfloor}\frac{\log^2 k}{k\log k}\lesssim  \log \lfloor e^{(m+1)^\alpha}\rfloor \log\Big( \exp \big((m+1)^\alpha-m^\alpha\big) \Big)\lesssim m^{\alpha}m^{\alpha-1} \\&\lesssim &m^{-1/5},
\end{eqnarray*}
 Azuma's inequality yields 
$$
\Prob (|Y'_m(x)|\geq\lambda)\lesssim \exp \Big(-c'\lambda^2\Big(\sum_{k=r_m}^{r_{m+1}-1}\frac{\log^2 p_k}{p_k}\Big)^{-1}\Big)
\lesssim \exp \big(-c'\lambda^2m^{1/5}\big).
$$
In particular, we  obtain that for some constants $c'', C$ that work for all $x\in (0,1)$ we have 
$$
\E \exp(c''m^{1/5}|Y'_m(x)|^2)\leq C.
$$
A similar estimate holds with $Y_m(x)$ in place of $Y'_m(x).$
As in the proof \eqref{eq:182}  (see \eqref{eq:sobolev}, \eqref{e21} and \eqref{e22}) we deduce that $\E \exp(c'''m^{1/5}\|Y_m(x)\|^2_{L^\infty(0,1)})\leq C$, and again   Levy's inequality enables us to gather that
\begin{equation}\label{eq221}
\Prob \big(\max_{r_m\leq u\leq r_{m+1}-1}\|\sum_{k=r_m}^{u}A_k\|_{L^\infty(0,1)}>\lambda\big) \lesssim \exp(-c'''m^{1/5}\lambda^2).
\end{equation}
Summing over $m$  yields for $\lambda \geq 1$
\begin{equation}\label{eq222}
\Prob \big(\sup_{m\geq 1}\max_{r_m\leq u\leq r_{m+1}-1}\|\sum_{k=r_m}^{u}A_k\|_{L^\infty(0,1)}>\lambda\big) \lesssim \sum_{m=1}^\infty \exp(-c'''m^{1/5}\lambda^2)\lesssim \exp(-c''''\lambda^2).
\end{equation}
Exactly  the same proof where Azuma is replaced by elementary estimates for Gaussian variables  yields the corresponding estimate  for our Gaussian approximation fields.  An easy Borel-Cantelli argument that uses estimates like \eqref{eq222} in combination with Proposition \ref{pr2}  then shows the existence of the uniform limit $E(x)=\lim_{N\to\infty}E_N(x).$  Finally, combining  \eqref{eq222} with \eqref {eq:204}  yields  \eqref{eq:204v}.
Together with our previous considerations this concludes the proof of the theorem.
\end{proof}

\section{Convergence to a chaos measure and multifractality in the subcritical case}

For a proper introduction to the theory of Gaussian multiplicative chaos, we refer the reader to Kahane's original work \cite{Kahane} or the recent review by Rhodes and Vargas \cite{RV}. For the convenience of the reader, we nevertheless recall the main results from the theory that are relevant to us. 

\begin{theorem}\label{th:gmc}
Assume that we have a sequence of independent Gaussian fields $(Y_k)_{k=1}^\infty$ on $[0,1]$ and the covariance kernel of $Y_k$ is $K_{Y_k}$, where $K_{Y_k}$ is continuous on $[0,1]$. Define the field

\begin{equation*}
X_n=\sum_{k=1}^n Y_k,
\end{equation*}

\noindent and assume that the covariance kernel $K_{X_n}$ converges as $n\to \infty$  locally uniformly in $[0,1]^2\setminus \{x=y\}$  to a function on $[0,1]^2$ which is of the form

\begin{equation*}
\log \frac{1}{|x-y|}+g(x,y),
\end{equation*}

\noindent where $g$ is bounded and continuous. Moreover, assume that   there is a constant $C<\infty$ so that
\begin{equation}\label{eq:ylaraja}
K_{X_n}(x,y)\leq \log \frac{1}{|x-y|}+ C \quad {\rm for \; all}\quad  x,y\in [0,1]\quad\mathrm {and}\quad n\geq 1.
\end{equation}
Then for $\beta>0$ the random measure 

\begin{equation*}
\nu_{\beta,n}(dx)=\frac{e^{\beta X_n(x)}}{\E e^{\beta X_n(x)}}dx
\end{equation*}

\noindent converges almost surely with respect to the topology of weak convergence of measures to a non-trivial limiting measure $\nu_\beta$. This limiting measure is a non-trivial random measure for $\beta<\beta_c=\sqrt{2}$ and for $\beta\geq \beta_c$, it is the zero measure. Moreover, if $0<\beta<\sqrt{2}$, and $0<p< 2/\beta^2$, then for a compact set $A\subset [0,1]$

\begin{equation*}
\E (\nu_\beta(A)^p)<\infty.
\end{equation*}

Also for $q\in[0,2/\beta^2)$

\begin{equation*}
\lim_{r\to 0}\frac{\log \E (\nu_\beta(B(x,r))^q)}{\log r}=(1+\frac{\beta^2}{2})q-\frac{\beta^2}{2}q^2.
\end{equation*}

\end{theorem}
\begin{proof} (Sketch) By \eqref{eq:ylaraja} and Kahane's convexity inequality (see \cite[Theorem 2.1]{RV}) one may easily compare to a standard approximation of a chaos measure and deduce  that for any $\beta<\beta_c$ the random variables $\nu_{\beta, n} ([0,1])$ form an $L^p$-martingale. At this stage the standard theory of multiplicative chaos can be applied to obtain the rest of the claims, see e.g.  \cite[Theorems 2.5,2.11, and 2.14]{RV}. 

\end{proof}

To apply Kahane's construction of a Gaussian multiplicative chaos measure, we'll need to establish that the covariance of our Gaussian field satisfies the requirements of Theorem \ref{th:gmc}. Let us introduce some notation for  the covariance of the $N$:th partial sum of the Gaussian approximation field
$$
G_N(x)=\sum_{j=1}^N\frac{1}{\sqrt{2p_j}}\left(W^{(1)}_j\cos(x\log p_j)+W^{(2)}_j\sin(x\log p_j)\right).
$$
A direct computation shows that
$$
K_{ G_n}(x-y):=\E G_n(x)G_n(y) =\psi_N(x-y),
$$
where 
$$
\psi_N(u):=\frac{1}{2}\sum_{j=1}^N\frac{\cos(u\log p_j)}{p_j}.
$$

The following result is enough for us to be able to apply Kahane's theory for defining a multiplicative chaos measure. It is of interest to note that we are dealing with a logarithmically correlated translation invariant field whose covariance deviates from $\frac{1}{2}\log (1/|x-y|)$ by only  a  smooth function.

\begin{lemma}\label{le:covariance_approximation}
We have 
$$
\Big| K_{ G_N}(x,y)-\frac{1}{2}\log\Big(\min \Big(\frac{1}{|x-y|},\log N \Big)\Big)\Big|\;\leq \; C,
$$
where $C$ is uniform over $n\geq 1$ and $(x,y)\in (0,1).$ Moreover, if $x\not=y$
$$
K_{ G_n}(x,y)\longrightarrow K_G(x,y)=\frac{1}{2}\log\left(\frac{1}{|x-y|}\right)+g(x-y)\quad\textrm{as}\quad {n\to\infty},
$$
with  local uniform convergence outside the diagonal. Moreover $g\in C^\infty (-2,2)$. A fortiori, the limit field $G$ is logarithmically correlated and translation invariant. 
\end{lemma}

Before proving the lemma, let us note that Theorem \ref{th: gaussian_appro} is a direct consequence of  this Lemma and Theorem \ref{th:g-approx}.

\begin{proof} We shall employ the notation where $z\add \widetilde{z}$ for given quantities $z=z_N(u),\widetilde z=\widetilde z_N(u)$ stands for the uniform inequality $|z_N(u)-\widetilde z_N(u)|\leq C$ with a universal bound $C$ and such that $\lim_{N\to\infty} (z_N(u)-\widetilde z_N(u))$ converges uniformly to a continuous function on the interval $u\in [-2,2].$ We shall employ the well-known asymptotics

\begin{equation}\label{eq:pr_as}
p_j=j\log j +O(j\log\log j).
\end{equation}
This implies that $\sum_{j=1}^\infty \frac{|\log p_j-\log (j\log j)|}{p_j}<\infty$ and since the cosine function is 1-Lipschitz we obtain
$$
\psi_N(u) \add \frac{1}{2}\sum_{j=1}^N\frac{\cos\big(u\log (j\log j)\big)}{p_j}.
$$
In a similar vein, $\sum_{j=1}^\infty \big|p_j^{-1}-(j\log j)^{-1}\big|<\infty$ which leads to
\begin{equation}\label{eq:700}
\psi_N(u) \add \frac{1}{2}\sum_{j=1}^N\frac{\cos\big(u\log (j\log j)\big)}{j\log j}.
\end{equation}
Next we observe that for all $u\in [-2,2]$ and $x\geq 10$
$$
 \left|\frac{d}{dx} \left(\frac{\cos\big(u\log (x\log x)\big)}{x\log x}\right)\right| 
 \;\leq \; \frac{6}{x^2\log x}.
$$
Since $\int_{10}^\infty (x^2\log x)^{-1}dx <\infty ,$ it follows that
\begin{equation}\label{eq:701}
\psi_N(u) \add \frac{1}{2}\int_{x=10}^N\frac{\cos\big(u\log (x\log x)\big)dx}{x\log x}.
\end{equation}
To continue, we note that
$$
\int_{x=10}^\infty\left| 1-\frac{1+\log x}{\log x+\log\log x}\right|\frac{dx}{x\log x}<\infty
$$
so that
\begin{eqnarray}\label{eq:702}
\psi_N(u) &\add& \frac{1}{2}\int_{x=10}^N\frac{\cos\big(u\log (x\log x)\big)}{\log(x\log x)}\frac{(1+\log x)dx}{x\log x}\nonumber\\
&\add& \frac{1}{2}\int_{1}^{\log N+\log\log N}\frac{\cos(ut)}{t}dt\add  \frac{1}{2}\int_{1}^{\log N}\frac{\cos(ut)}{t}dt\\
&=& \frac{1}{2}\int_{u}^{u\log N}\frac{\cos(x)}{x}dx\; =:\; A(u,N).\nonumber
\end{eqnarray}
Above in the first step we performed the change of variables $u=\log(x\log x)$ and noted that $du=(1+\log x)dx/x\log x$.  In the second to last step noted that $\int_{\log N}^{\log N+\log\log N}t^{-1}dt =o(1)$ as $N\to\infty.$

It remains to prove the claim for $A(u, N)$ defined in \eqref{eq:702}. 
Since the limit $\lim_{z\to\infty} \int_{1}^{z}\frac{\cos(x)}{x}dx$  exists and is finite, we see
directly from the definition that for any $\varepsilon_0 >0$ in the set $\{ \varepsilon_0\leq |u|\leq 2\}$ the function $A(u,N)$ converges uniformly to a continuous function of $u$ as $N\to\infty$. Moreover, since $\int_0^1|\cos (x)-1|x^{-1}dx<\infty,$ we get  for $|u|\geq (\log N)^{-1}$
$$
|A(u,N)-\int_u^1x^{-1}dx|=|A(u,N)-\log(1/u)|\leq C,
$$
where $C$ is independent of $N$ and $u$. 
Finally, if $|u|\leq(\log N)^{-1}$ we get in a similar manner
$$
|A(u,N)-\int_u^{u\log N}x^{-1}dx|=|A(u,N)-\log\log N|\leq C',
$$
and now $C'$ is independent of $N$ and $u\in \{|u|\geq (\log N)^{-1}\}.$ This proves the first statement of the lemma.

By \eqref{eq:702} we deduce that there is a continuous function $\widetilde b(u)$ on $[-2,2]$ so that the limit $\psi$ of the functions $\psi_N$ takes the form 
\begin{eqnarray}\label{eq:704}
\psi (u)= \widetilde b(u)+ \frac{1}{2}\int_{u}^{\infty}\frac{\cos(t)}{t}dt{u} =
\frac{1}{2}\log\left(|u|^{-1}\right)+ b(u)\quad\textrm{for}\quad 0<|u|<2,
\end{eqnarray}
with $b\in C([-2,2])$  as $u\mapsto \int_0^u(\cos (x)-1)x^{-1}dx$ is continuous over $x\in [-2,2].$  Especially, we know that $\psi (x-y)$ yields the covariance operator of our limit field since the estimates we have proven show that $\psi_N(x-y)\mapsto\psi(x-y)$ in $L^2([0,1]^2)$, and convergence in the Hilbert-Schmidt norm is enough to identify the limit covariance of a sequence of Gaussian fields converging a.s. in the sense of distributions. We still want to upgrade $b$ to be smooth. For that end we first fix $\delta_0 >0$ and observe that what we have proved up to now (see especially \eqref{eq:702} ) yields that  we have 
\begin{equation}\label{eq:705}
\psi (u)=\frac{1}{2} {\rm Re\,}\left(\lim_{N\to\infty} \sum_{j=1}^Np_j^{-1-iu}\right)
\end{equation}
with uniform convergence in the set $\{ \delta_0\leq |u|\leq 2\}$. However, if we apply exactly the same argument as above to the sum ${\rm Re\,}\big(\sum_{j=1}^Np^{-1-\varepsilon -iu}\big)$
for, say, $\varepsilon\in [0,1/2],$ we obtain uniform (in $\varepsilon$) estimates for the convergence of the series
$$
{\rm Re\,}\left(\sum_{j=1}^\infty p_j^{-1-\varepsilon -iu}\right)
$$
for any fixed $u\in (0,2)$.
Especially, we deduce by invoking the logarithm of the Euler product of the Riemann zeta function that 
\begin{eqnarray}\label{eq:706}
\psi (u)&=&\lim_{\varepsilon \to 0^+} \frac{1}{2}{\rm Re\,}\left(\sum_{j=1}^\infty p_j^{-1-\varepsilon -iu}\right)\\
&=&\lim_{\varepsilon \to 0^+} \frac{1}{2}{\rm Re\,}\left(\zeta(1+\varepsilon+iu)-\sum_{k=2}^\infty\sum_{j=1}^\infty k^{-1}p_j^{-k(1+\varepsilon +iu)}\right)\nonumber\\
&=& \frac{1}{2}{\rm Re\,}\left(\zeta(1+iu)-\sum_{k=2}^\infty\sum_{j=1}^\infty k^{-1}p_j^{-k(1+iu)}\right),\nonumber\\
&=:& \frac{1}{2}{\rm Re\,}\left(\log(\zeta(1+iu))-A(u)\right),\nonumber
\end{eqnarray}
as the last written double sum converges absolutely (uniformly in $\varepsilon$).
It remains to note that $\log(\zeta(1+iu))$ is real analytic on $(0,\infty)$, and the function $A$ is $C^\infty$-smooth on the same set as term wise differentiation of $A$ $\ell$ times with respect to $u$ produces a series with the majorant series
$$
\sum_p\sum_{k=2}^\infty k^{\ell-1}p^{-k}\log^\ell p\leq
\sum_pp^{-3/2}\Big(\sum_{r=0}^\infty (r+2)^{\ell-1}p^{-r}\Big)<\infty .
$$
\end{proof}

Before proving the convergence of the subcritical chaos we still need to note that the expectation of the exponential martingale obtained via the Gaussian approximation converges (apart from a multiplicative constant) with the same rate as that of our original exponential martingale. 
\begin{lemma}\label{le:exp} For any $\beta >0$ there is a constant $C=C(\beta)$ such that
$$
C^{-1}\E \exp(\beta G_N)\leq \E \exp(\beta X_N) \leq C\E \exp(\beta G_N)\quad \textrm{for all }\quad N\geq 1.
$$
\end{lemma}
\begin{proof}
Note first that there is an $a_0>0$ such that for arbitrary $y \in\R$ the asymptotics of the Laplace transform satisfy:
$$
\E \exp \big(\lambda (\cos(\theta_p)\cos (y)+\sin(\theta_p)\sin (y))\big)=\exp(\frac{1}{4}\lambda^2+O(\lambda^3)) \quad\textrm{for}\quad |\lambda|\leq a_0.
$$
This can be seen by noting that the Laplace transform is analytic, symmetric and has second derivative equal to 1 at zero since 
$\E\big(\cos(\theta_p)\cos (y)+\sin(\theta_p)\sin (y)\big)^2 =1/2$ for all $y$. Since $(p_k)^{-1/2}\to 0$ as $k\to\infty$, and $\sum_{k=1}^\infty ((p_k)^{-1/2})^3<\infty,$ we may apply independence and the above asymptotics for large enough $k$ (depending on $\beta$)
to deduce that 
$$
\E \exp(\beta X_N)\approx \exp\big(\frac{\beta^2}{4}\sum_{k=1}^N\frac{1}{p_k}\big)=\E \exp(\beta G_N).
$$
\end{proof}

\begin{remark}
Note that in our case the asymptotic covariance has a singularity of the form $-\frac{1}{2}\log |x-y|$ instead of $-\log |x-y|$ as in Theorem \ref{th:gmc}. This simply means that we replace $\beta$ by $\beta/\sqrt{2}$ in Theorem \ref{th:gmc}.
\end{remark}

One should note that combining the above lemmas  we see that
\begin{equation}\label{eq:expmoments}
\E \exp(\beta X_N) \approx \E \exp(\beta G_N)\approx \exp(\frac{\beta^2}{4}\log\log N)=(\log N)^{\beta^2/4}\quad \textrm{for}\quad N\geq 1.
\end{equation}

Finally we are ready for:

\begin{proof}[Proof of Theorem \ref{th:main}] Consider the Gaussian field $G$ that is the limit of the fields $G_N$. For $\beta<2$ the corresponding log-normal chaos exists due to Theorem \ref{th:gmc}, and the approximating measures obtained from the fields $G_N$ converge  to $\nu_\beta$. Especially, there is a $\widetilde p>1$ such that $\nu_\beta$ satisfies
$\E (\nu_{\beta, N}(0,1)^{\widetilde p})\leq C<\infty$ for all $N\geq 1$. Recall that we want to prove that for each continuous $f:[0,1]\to \R_+$, $\mu_{N,\beta}(f)$ converges almost surely to a non-trivial random variable. By the construction of the measure, this is a positive martingale, so it is enough to prove that it is bounded in $L^p$ for some $p>1$. For this it is then enough to show that $\mu_{N,\beta}(0,1)$ is bounded in $L^p$ for some $p>1$. Choose $p\in (1,\widetilde p)$ and consider the approximating measures $\mu_{\beta,N}$ corresponding to the fields $X_N(x)$. Since the normalisation factors are comparable, we obtain by H\"older's inequality
and Theorem \ref{th:g-approx}
\begin{align*}
\E\mu_{\beta,N}(0,1)^p&\leq \E\Big( \exp(p\beta \| E_N\|_{L^\infty(0,1)})(\nu_{\beta,N}(0,1))^p\Big) \\
&\leq \left(\E \exp\Big(p(\widetilde p/p)'\beta \| E_N\|_{L^\infty(0,1)}\Big)\right)^{1/(\widetilde p/p)'}
\big(\E(\nu_{\beta,N}(0,1))^{\widetilde p}\big)^{p/\widetilde p}\leq C',
\end{align*}
where $'$ denotes the H\"older conjugation. This yields uniform integrability of $\mu_{N,\beta}(0,1)$ which proves the existence of a non-trivial limit. The second claim is then a direct consequence of Theorem \ref{th:g-approx}. 

For $\beta\geq \beta_c$, we see similarly using Theorem \ref{th:gmc} and Theorem \ref{th:g-approx} that $\mu_{\beta,N}$ converges to zero since $\nu_{\beta,N}$ converges to zero.
\end{proof}

We can also immediately prove Proposition \ref{prop:multif}.

\begin{proof}[Proof of Proposition \ref{prop:multif}]
As in our proof that the martingale $(\mu_{\beta,N})_N$ is uniformly integrable, we note that for $0<q<\widetilde p$,  and for any $x\in (0,1)$ and $r>0$

\begin{equation}
\lim_{N\to\infty}\E(\mu_{\beta,N}(B(x,r))^{q})=\E(\mu_{\beta}(B(x,r))^{q}).
\end{equation}

Let us first note that 

\begin{equation}
e^{-\beta||E_{N}||_{L^{\infty}(0,1)}}\leq \frac{\mu_{\beta,N}(B(x,r))}{\int_{x-r}^{x+r} \frac{e^{\beta G_N(y)}}{\E e^{\beta X_N(y)}}dy}\leq e^{\beta||E_{N}||_{L^{\infty}(0,1)}}.
\end{equation}

Then take $\epsilon>0$ so small that $(1+\epsilon)q<q_c$. Arguing as in the proof of Theorem \ref{th:main} with H\"older's inequality we have for some constant $C>0$

\begin{equation}
\E (\mu_{\beta}(B(x,r))^{q})\leq C\left(\E \nu_{\beta}(B(x,r))^{q(1+\epsilon)}\right)^{\frac{1}{1+\epsilon}}.
\end{equation}

As the covariance of the limiting Gaussian field is of the form $-\frac{1}{2}\log |x-y|+g(x-y)$, we know how the expectation here scales in $r$ (see e.g. \cite[Theorem 2.14]{RV}):

\begin{equation}
\E \nu_{\beta}(B(x,r))^{q(1+\epsilon)}\sim r^{(1+\beta^{2}/2)q(1+\epsilon)-(q(1+\epsilon))^{2}\beta^{2}/2}.
\end{equation}

Taking logarithms, dividing by $\log r$,  letting $r\to 0$ and then $\epsilon\to 0$ we get the correct upper bound. 

\vspace{0.3cm}

For the lower bound, we use the reverse H\"older inequality: let $p>1$, $f$ and $g$ be measurable such that  $g\neq 0$ almost surely. Then 

\begin{equation}
\E|fg|\geq \left(\E|f|^{1/p}\right)^{p}\left(\E|g|^{-1/(p-1)}\right)^{-(p-1)}.
\end{equation}

With a similar argument simply replacing H\"older's inequality by the reverse H\"older inequality we find for some $C=C(\beta,q,\epsilon)>0$

\begin{equation}
\E(\mu_\beta(B(x,r))^{q})\geq C\left(\E \nu_{\beta}(B(x,r))^{q/(1+\epsilon)}\right)^{1+\epsilon}.
\end{equation}

Performing the same steps as above we get the lower bound as well.

\end{proof}

\section{The critical measure}

In this section we establish the existence of the critical measure. We'll do this by showing that $G_N(x)=\widetilde{G}_N(x)+D_N(x)$, where $D_N$ converges almost surely to a nice continuous Gaussian field and $\widetilde{G}_N$ is sequence of  Gaussian fields for which the critical measure can be shown to exist (using results from \cite{JS}). More precisely, the result we'll need is:

\begin{theorem}[{\cite[Theorem 1.1]{JS}}]\label{th:js1}
Let $(X_N)$ and $(\widetilde{X}_N)$ be two sequences of H\"older regular Gaussian fields on $[0,1]$ $($that is, $(x,y)\to \sqrt{\E(X_N(x)-X_N(y))^2}$ is H\"older continuous on $[0,1]^2$ $)$. Assume that $A_Ne^{\widetilde{X}_N(x)-\frac{1}{2}\E \widetilde{X}_N(x)^2}dx$ converges weakly in distribution to an almost surely non-ato-mic measure $\widetilde{\mu}$, where $A_N$ is a deterministic scalar sequence. Assume further that the covariances $C_N(x,y)=\E X_N(x)X_N(y)$ and $\widetilde{C}_N(x,y)=\E \widetilde{X}_N(x) \widetilde{X}_N(y)$ satisfy the following conditions: there exists a constant $K\in(0,\infty)$ (independent of $N$) such that for all $N\geq 1$, 

\begin{equation}
\sup_{x,y\in[0,1]}|C_N(x,y)-\widetilde{C}_N(x,y)|\leq K
\end{equation}

\noindent and for each $\delta>0$

\begin{equation}
\lim_{N\to\infty}\sup_{|x-y|>\delta}|C_N(x,y)-\widetilde{C}_N(x,y)|=0.
\end{equation}

Then also $A_ne^{X_N(x)-\frac{1}{2}\E X_N(x)^2}dx$ converges weakly to $\widetilde{\mu}$.
\end{theorem}

To do this, we thus need a reference approximation for which convergence is known, and a representation of our field which gives us good control on the covariance. Let us first discuss the reference field.

For this, we recall a construction from \cite{BM} and make use of results in \cite{DRSV1}. 

\begin{definition}{\rm
Let $W$ denote a white noise on $\R\times [-1/2,3/2]$. For $t\in \R$ and $x\in[0,1]$, write

\begin{equation}
\widetilde{G}_t(x)=\frac{1}{\sqrt{2}}\int_{-\infty}^t \int_{-1/2}^{3/2}\mathbf{1}\left\lbrace |x-y|\leq \frac{1}{2}\min(e^{-s},1)\right\rbrace e^{s/2}W(ds,dy).
\end{equation}}
\end{definition}

The covariance of the field is

\begin{equation}
\E\left(\widetilde{G}_t(x)\widetilde{G}_t(y)\right)=\begin{cases}
\frac{1}{2}\left[1+t-e^t |x-y|\right], & |x-y|\leq e^{-t}\\
-\frac{1}{2}\log |x-y|,  & e^{-t}\leq |x-y|\leq 1
\end{cases}.
\end{equation}

\noindent Obviously the above field is H\"older-regular as it is $C^1$. As pointed out in \cite[Remark 3]{DRSV1}, the main results of \cite{DRSV1} apply also to the measure $\sqrt{t}e^{2\widetilde{G}_t(x)-2\E \widetilde{G}_t(x)^2}dx$ , whence it converges weakly in probability to a non-trivial, and non-atomic random measure, as $t\to \infty$.

Our next task is to then approximate our field by one whose covariance we can control. We'll carry this out in several steps. While perhaps the results we need might follow from general results for Gaussian processes, we will repeat a variation of our argument in Lemma \ref{le:E2} several times. Our first step is to consider a more concrete sum - we replace the summation over primes by a more regular one in terms of the Logarithmic integral: define 

\begin{equation}
G_{N,1}(x)=\sum_{j=1}^N\frac{1}{\sqrt{2 \mathrm{Li}^{-1}(j)}}\big[W_j^{(1)}\cos(x\log \mathrm{Li}^{-1}(j))+W_j^{(2)}\sin(x\log \mathrm{Li}^{-1}(j))\big].
\end{equation}

Let us show that this is a good approximation to $G_N$.

\begin{lemma}
There exists a random continuous function $F_1:[0,1]\to\R$ such that almost surely, $G_{N,1}-G_N$ converges to $F_1$ uniformly.
\end{lemma}

\begin{proof}
Let us write $F_{N,1}=G_{N,1}-G_N$. Our argument is very similar to the proof of Lemma \ref{le:E2}. Due to this, we won't go through all of the details. Again, it will be enough to estimate $\E||F_{N,1}-F_{M,1}||_{L^\infty(0,1)}^2$ and we'll do this by making use of the fact that we can bound the $\sup$-norm by the Sobolev norm in our case. For the Sobolev norm, we note that it follows from Lemma \ref{le:li-inverse} that 

\begin{align}
\left|\frac{1}{\sqrt{2 p_j}}\cos(x\log p_j)-\frac{1}{\sqrt{2\mathrm{Li}^{-1}(j)}}\cos(x\log \mathrm{Li}^{-1}(j))\right|&\lesssim \frac{|p_j-\mathrm{Li}^{-1}(j)|}{p_j^{3/2}}\\
\notag &\lesssim \frac{j e^{-\sqrt{\log j}}}{(j\log j)^{3/2}}.
\end{align}

A similar estimate holds for the sine-term. Differentiating only gives an extra $\log j$ here. So we see that if $M< N$, then using the Sobolev bound one finds

\begin{equation}
\E||F_{N,1}-F_{M,1}||_{L^\infty(0,1)}^2\lesssim \sum_{j=M+1}^{N}\frac{e^{-2\sqrt{\log j}}}{j\log j},
\end{equation}

\noindent which is bounded in $N$ and $M$. We can then proceed as in Lemma \ref{le:E2},  using again L\'evy's inequality. 
\end{proof}

Next we find it useful to move to the continuous Fourier side and perform further smoothing there. We'll do this by first replacing the Gaussian blocks by Wiener integrals. More precisely, consider $B_t^{(1)}$ and $B_t^{(2)}$ two independent Brownian motions, and let us assume that the Gaussian variables $W^{(i)}_j$ are constructed from $B_t^{(i)}$ in the following manner:

\begin{equation}
W_j^{(i)}=\int_{\mathrm{Li}^{-1}(j)}^{\mathrm{Li}^{-1}(j+1)}\frac{dB_t^{(i)}}{\sqrt{\mathrm{Li}^{-1}(j+1)-\mathrm{Li}^{-1}(j)}}.
\end{equation}

First of all, we claim the following:

\begin{lemma}
Let 

\begin{align}
G_{N,2}(x)&=\sum_{j=1}^N \int_{\mathrm{Li}^{-1}(j)}^{\mathrm{Li}^{-1}(j+1)}\frac{\cos(x\log t)}{\sqrt{2t}}\frac{dB_t^{(1)}}{\sqrt{\mathrm{Li}^{-1}(j+1)-\mathrm{Li}^{-1}(j)}}\\
\notag &\quad +\sum_{j=1}^N \int_{\mathrm{Li}^{-1}(j)}^{\mathrm{Li}^{-1}(j+1)}\frac{\sin(x\log t)}{\sqrt{2t}}\frac{dB_t^{(2)}}{\sqrt{\mathrm{Li}^{-1}(j+1)-\mathrm{Li}^{-1}(j)}}.
\end{align}

Then almost surely, $G_{N,2}-G_{N,1}$ converges uniformly to a continuous function $F_{2}$.
\end{lemma}

\begin{proof}
This is very similar to the previous lemma, and again we'll follow the proof of Lemma \ref{le:E2}. By Ito's isometry, to get a hold of the expectation of the square of the Sobolev norm, we now only need to estimate 

\begin{equation}
\frac{1}{\mathrm{Li}^{-1}(j+1)-\mathrm{Li}^{-1}(j)}\int_{\mathrm{Li}^{-1}(j)}^{\mathrm{Li}^{-1}(j+1)}\left[\frac{\cos(x\log \mathrm{Li}^{-1}(j))}{\sqrt{2\mathrm{Li}^{-1}(j)}}-\frac{\cos(x\log t)}{\sqrt{2t}}\right]^2dt,
\end{equation}

\noindent and a similar derivative term. The integral above is  ${\mathcal O}\big(\frac{\mathrm{Li}^{-1}(j+1)-\mathrm{Li}^{-1} (j) }{ \mathrm{Li}^{-1}(j)^{3/2}}\big)={\mathcal O}\big(j^{-3/2}\log^{-1/2}(j)\big)$, while the derivative term comes with an extra $\log^2 j$. Both of these are summable over $j$, so we can conclude as before.
\end{proof}

To proceed, we'll want to replace the $1/\sqrt{\mathrm{Li}^{-1}(j+1)-\mathrm{Li}^{-1}(j)}$ by something more convenient. More precisely, we'll make use of the following approximation.

\begin{lemma}
Let 

\begin{equation}
G_{N,3}(x)=\int_{\mathrm{Li}^{-1}(1)}^{\mathrm{Li}^{-1}(N+1)}\frac{\cos(x\log t)}{\sqrt{2t}}\frac{dB_t^{(1)}}{\sqrt{\log t}}+\int_{\mathrm{Li}^{-1}(1)}^{\mathrm{Li}^{-1}(N+1)}\frac{\sin(x\log t)}{\sqrt{2t}}\frac{dB_t^{(2)}}{\sqrt{\log t}}.
\end{equation}

Then almost surely, as $N\to\infty$, $G_{N,3}-G_{N,2}$ converges uniformly to a random continuous function $F_{3}$.
\end{lemma}

\begin{proof}
Again the reasoning is an in Lemma \ref{le:E2}. Now we need to estimate terms of the form 

\begin{equation}
\int_{\mathrm{Li}^{-1}(j)}^{\mathrm{Li}^{-1}(j+1)}\left[\frac{1}{\sqrt{\mathrm{Li}^{-1}(j+1)-\mathrm{Li}^{-1}(j)}}-\frac{1}{\sqrt{\log t}}\right]^2\frac{\cos^2(x\log t)}{2t}dt,
\end{equation}

\noindent and corresponding ones with a sine or similar ones coming with a factor of $\log^2 t$ coming from the derivative term in the Sobolev estimate. To estimate such a term, we see that it is enough for us to estimate the difference $|\mathrm{Li}^{-1}(j+1)-\mathrm{Li}^{-1}(j)-\log t|$ for $t\in[\mathrm{Li}^{-1}(j),\mathrm{Li}^{-1}(j+1)]$. For this, we note that 

\begin{align}
\mathrm{Li}^{-1}(j+1)-\mathrm{Li}^{-1}(j)&=\int_{\mathrm{Li}^{-1}(j)}^{\mathrm{Li}^{-1}(j+1)}dt
=\int_{j}^{j+1}\log( \mathrm{Li}^{-1}(s))ds,
\end{align}

\noindent where we made the change of variable $t=\mathrm{Li}^{-1}(s)$, and used the fact that $\mathrm{Li}'(x)=1/\log x$. Due to the asymptotics $\mathrm{Li}^{-1}(j)\sim j\log j$  and $(\mathrm{Li}^{-1})'(j)\sim \log j$ we thus have for $t\in[\mathrm{Li}^{-1}(j),\mathrm{Li}^{-1}(j+1)]$

\begin{equation}
\left|\mathrm{Li}^{-1}(j+1)-\mathrm{Li}^{-1}(j)-\log t\right|\leq \log \frac{\mathrm{Li}^{-1}(j+1)}{\mathrm{Li}^{-1}(j)}\lesssim \frac{\mathrm{Li}^{-1}(j+1)-\mathrm{Li}^{-1}(j)}{\mathrm{Li}^{-1}(j)}\lesssim j^{-1}.
\end{equation}

Hence the square of the Sobolev norm can be bounded by ${\mathcal O}\big((j\log j)^{-3}\big)$, which is summable and the rest of the proof goes through as before.
\end{proof}

We note that  $e^{-s/2}dB_{e^s}^{(i)}= d\widetilde B^{i}_{s}$ are standard independent Brownian motions. After performing a change of variables in the integral we thus obtain 

\begin{align}
G_{N,3}(x)&=\int_{\log \mathrm{Li}^{-1}(1)}^{\log \mathrm{Li}^{-1}(N+1)}\frac{\cos(x s)}{\sqrt{s}}\frac{dB_{e^s}^{(1)}}{\sqrt{2}e^{s/2}}+\int_{\log \mathrm{Li}^{-1}(1)}^{\log \mathrm{Li}^{-1}(N+1)}\frac{\sin(xs)}{\sqrt{s}}\frac{dB_{e^s}^{(2)}}{\sqrt{2}e^{s/2}}\\
&=2^{-1/2}\int_{\log \mathrm{Li}^{-1}(1)}^{\log \mathrm{Li}^{-1}(N+1)}\frac{\cos(x s)}{\sqrt{s}}d\widetilde{B}_{s}^{(1)}+2^{-1/2}\int_{\log \mathrm{Li}^{-1}(1)}^{\log \mathrm{Li}^{-1}(N+1)}\frac{\sin(xs)}{\sqrt{s}}d\widetilde{B}_{s}^{(2)}
\end{align}

We now want to replace $1/\sqrt{s}$ by something that will allow us to reach the desired covariance in the limit. Let us consider the translation invariant covariance, already alluded to before, that is induced by the function $C(x)=\max(-\log |x|,0)$. Then

\begin{align}
\widehat{C}(k)&=\int_{-1}^1 e^{ikx}\log \frac{1}{|x|}dx\\
\notag &=2\int_0^1 \cos(kx)\log \frac{1}{x}dx\\
\notag &= \frac{2}{k}\int_0^k \cos y\log \frac{k}{y}dy\\
\notag &=\frac{2}{k}\int_0^k\frac{\sin y}{y}dy,
\end{align}

\noindent where in the last step we integrated by parts. This is positive (as it should since it's the Fourier transform of a translation invariant covariance), and as $k\to\infty$, it behaves likes $\pi/k+\mathcal{O}(k^{-2})$. Thus it should be possible to replace $1/\sqrt{s}$ in our field by $\sqrt{\widehat{C}(s)}/\sqrt{\pi}$, which will turn out to be precisely what we need.

\begin{lemma}
Let 

\begin{align}
G_{N,4}(x)&=\frac{1}{\sqrt{\pi}}\int_{\log \mathrm{Li}^{-1}(1)}^{\log \mathrm{Li}^{-1}(N+1)}\sqrt{\widehat{C}(s)}\cos(xs)\frac{dB_{e^s}^{(1)}}{\sqrt{2}e^{s/2}}+\frac{1}{\sqrt{\pi}}\int_{\log \mathrm{Li}^{-1}(1)}^{\log \mathrm{Li}^{-1}(N+1)}\sqrt{\widehat{C}(s)}\sin(xs)\frac{dB_{e^s}^{(2)}}{\sqrt{2}e^{s/2}}.
\end{align}

Then almost surely, $G_{N,4}-G_{N,3}$ converges uniformly to a random continuous function $F_4$.
\end{lemma}

\begin{proof}
In this case, making use of the same Sobolev estimate as before would lead to a non-summable series, but we still can proceed by employing the following simple lemma

\begin{lemma}\label{le:frac-sobo}
Assume that $B(\xi )_\xi$ is a standard $($two-sided$)$ Brownian motion. Let $g:\R\to\C$ be a bounded measurable function with compact support. Let us denote by
$$
F(x):={\mathcal F}^{-1}\big(g(\cdot)dB(\cdot)\big)(x)=\frac{1}{2\pi}\int_{\R}e^{ix\xi}g(\xi)dB(\xi)
$$
the inverse Fourier transform of the $($almost-surely well-defined$)$ compactly supported distribution $g(\xi)dB(\xi)$.
Then for any $r>1/2$ we have
$$
\E \|F\|_{L^{\infty} (0,1)}^2 \lesssim \int_\R|g(\xi)|^2(1+|\xi|^2)^rd\xi.
$$
\end{lemma}
\begin{proof}
Let us first note that for, say smooth Schwartz test functions we obtain by Cauchy-Schwartz
\begin{equation}\label{eq:upotus1}
\|f\|_{L^\infty(\R)}\lesssim\|\widehat f\|_{L^1}\lesssim \|\widehat f(\xi) (1+|\xi|^2)^{r/2}\|_{L^2(\R)}\nonumber
\end{equation}
since $\|(1+|\cdot|^2)^{-r/2}\|_2<\infty$ for $r>1/2$ (actually this yields a proof of a special case of the Sobolev embedding theorem, see e.g. \cite[Theorem 6.2.4]{G}). In order to localize in the case where $f$ is smooth but not compactly supported, we  pick a real-valued and symmetric Schwartz test function $\phi$ with supp$(\phi)\subset[-1,1]$. We demand further that ${\mathcal F}^{-1}\phi(x)\geq 1/2$ on $[0,1]$.
We then observe that by the previous inequality
\begin{equation}\label{eq:upotus2}
\|f\|_{L^\infty(0,1)}\lesssim\|[{\mathcal F}^{-1}\phi ]f\|_{L^\infty(\R)}\lesssim\|\phi*\widehat f (\xi) (1+|\xi|^2)^{r/2}\|_{L^2(\R)}.
\end{equation}
Observe next that for any $\xi\in\R$
\begin{align}
\E \big|(gdB)*\phi(\xi)\big|^2  &=
\E\int_R\int_\R g(u)\phi(\xi-u)\overline{g(u')}\phi(\xi- u')dB(u)dB(u') \\ 
&=\int_\R|g(u)|^2\phi^2 (\xi-u)du= \big(|g|^2*\phi^2)(\xi).
\end{align}
By combing this with \eqref{eq:upotus2} it follows that
\begin{align}
\E\|F\|^2_{L^\infty(0,1)}&\lesssim\int_\R\big(|g|^2*\phi^2)(\xi)(1+|\xi|^2)^rd\xi= \int_\R\big(|g(\xi)|^2[(1+|\cdot|^2)^r*\phi^2](\xi)d\xi,
\end{align}
and the claim follows by noting that trivially $[(1+|\cdot|^2)^r*\phi^2](\xi)\lesssim (1+|\xi|^2)^r.$

\end{proof}

In our case, if we define $F_{N,4}=G_{N,4}-G_{N,3}$, an application of the above lemma with the choice $r=3/4$ results in the bound (for say $M\leq N$)

\begin{equation}
\E||F_{N,4}-F_{M,4}||_{L^\infty(0,1)}^2\lesssim \int_{\log \mathrm{Li}^{-1}(M+1)}^{\log \mathrm{Li}^{-1}(N+1)} (1+s^2)^{3/4}\left[\sqrt{\frac{\widehat{C}(s)}{\pi}}-\frac{1}{\sqrt{s}}\right]^2ds
\end{equation}
Note that
\begin{align}
\left|\sqrt{\frac{\widehat{C}(s)}{\pi}}-\frac{1}{\sqrt{s}}\right|&=\frac{1}{\sqrt{s}}\left|\sqrt{\frac{2}{\pi}\int_0^{s}\frac{\sin y}{y}dy}-1\right|
\notag \leq \frac{1}{\sqrt{s}}\frac{2}{\pi}\int_{s}^\infty \frac{\sin y}{y}dy =\mathcal{O}(s^{-3/2}),
\end{align}
where we made use of the fact that $\frac{2}{\pi}\int_0^\infty \frac{\sin y}{y}dy=1$ and the already mentioned asymptotic bound 
$
\int_{s}^\infty \frac{\sin y}{y}dy=\mathcal{O}(s^{-1}). 
$
It follows that 
\begin{equation}
\E||F_{N,4}-F_{M,4}||_{L^\infty(0,1)}^2\lesssim \int_{\log \mathrm{Li}^{-1}(M+1)}^{\log \mathrm{Li}^{-1}(N+1)} (1+s^2)^{3/4}s^{-3}ds,
\end{equation}
which is bounded in $N$ and $M$, so we proceed as before.
\end{proof}

To make use of Theorem \ref{th:js1} and compare $G_{N,4}$ to $\widetilde{G}_t$, we should see how $N$ and $t$ are related. To do this, let us calculate the variance of $G_{N,4}$ and require it to be $\frac{1}{2}t+\mathcal{O}(1)$. We have 

\begin{align}
\E G_{N,4}(x)^2&=\frac{1}{2\pi}\int_{\log \mathrm{Li}^{-1}(1)}^{\log \mathrm{Li}^{-1}(N+1)}\widehat{C}(s)ds\\
\notag &=\frac{1}{2}\int_{\log \mathrm{Li}^{-1}(1)}^{\log \mathrm{Li}^{-1}(N+1)}\frac{1}{s}ds+\int_{\log \mathrm{Li}^{-1}(1)}^{\log \mathrm{Li}^{-1}(N+1)}\mathcal{O}(s^{-2})ds\\
\notag &=\frac{1}{2}\log\log\mathrm{Li}^{-1}(N+1)+\mathcal{O}(1),
\end{align}

\noindent where we used the expansion of $\widehat{C}(s)$. Thus we should expect that $t=\log\log \mathrm{Li}^{-1}(N+1)$ should give a good estimate for the covariances. Indeed, for $|x-y|\leq 1/\log \mathrm{Li}^{-1}(N+1)$, we have 

\begin{align}
\E G_{N,4}(x)G_{N,4}(y)&=\frac{1}{2}\int_1^{\log\mathrm{Li}^{-1}(N+1)}\frac{1}{s}\cos(s|x-y|)ds+\mathcal{O}(1)\\
\notag &=\frac{1}{2}\int_{|x-y|}^{|x-y|\log \mathrm{Li}^{-1}(N+1)}\frac{1}{s}\cos sds+\mathcal{O}(1)\\
\notag &=\frac{1}{2}\int_{|x-y|}^{|x-y|\log \mathrm{Li}^{-1}(N+1)}\frac{1}{s}ds+\frac{1}{2}\int_{|x-y|}^{|x-y|\log \mathrm{Li}^{-1}(N+1)}\frac{\cos s-1}{s}ds+\mathcal{O}(1)\\
\notag &=\frac{1}{2}\log \log \mathrm{Li}^{-1}(N+1)+\mathcal{O}(1)
\end{align}

\noindent where the $\mathcal{O}(1)$ terms are uniform in $x,y$. For $|x-y|\geq 1/\log \mathrm{Li}^{-1}(N+1)$, elementary calculations show that

\begin{align}
\E G_{N,4}(x)G_{N,4}(y)&=\frac{1}{2}C(x-y)+\frac{1}{2}\int_{|x-y|[\log\mathrm{Li}^{-1}(N+1)+1]}^\infty \frac{\cos s}{s}ds+\mathit{o}(1),
\end{align}

\noindent where the $\mathit{o}(1)$ term is uniform in $x,y$. From this we see that for $C_N(x,y)=\E G_{N,4}(x)G_{N,4}(y)$ and $\widetilde{C}_N(x,y)=\E \widetilde{G}_t(x)\widetilde{G}_t(y)$ with $t=\log\log\mathrm{Li}^{-1}(N+1)$, the conditions on the distances between the covariances in Theorem \ref{th:js1} are satisfied. Let us finally note that all our approximating  fields are smooth with smooth, and especially they have H\"older covariances.

Before finishing our proof, we'll recall a further result we need from \cite{JS}.

\begin{lemma}[{\cite[Lemma 4.2 (ii)]{JS}}]\label{le:js}
Let $X$ be a H\"older regular Gaussian field on $[0,1]$ and assume that it is independent of the sequence of measures $(\mu_n)$ on $[0,1]$. If $e^{X}\mu_n$ converges weakly in distribution, then $\mu_n$ does as well.
\end{lemma}

We now turn to the proof.

\begin{proof}[Proof of Theorem \ref{th:critical}]
Let us introduce some notation. For $M\geq 0$, let 

\begin{equation}
\nu_{\beta_c,M, N}(dx)=\sqrt{\log\log \mathrm{Li}^{-1}(N+1)}e^{\beta_c (G_{N,4}(x)-G_{M,4}(x))-\frac{\beta_c^2}{2}\E[ G_{N,4}(x)^2-G_{M,4}(x)^2]}dx,
\end{equation}

\noindent where $G_{0,4}=0$. From Theorem \ref{th:js1} we see that $\nu_{\beta_c,0,N}$ converges weakly in distribution (to a non-trivial random measure) as $N\to\infty$. Then from Lemma \ref{le:js} we see that also $\nu_{\beta_c,M,N}$ converges weakly in law for any fixed $M\geq 0$. We also note that 
$
\mu_{\beta_c,N}(dx)=e^{f_N(x)}\nu_{\beta_c,0,N}(dx)
$, where $f_N$ is a sequence of  continuous functions converging uniformly almost surely to a continuous function $f$ and, by construction, $f_M$ is independent from $f_N-f_M$ for each $0\leq M<N$. Recall that we want to show that for each non-negative continuous $g:[0,1]\to[0,\infty)$, $\mu_{\beta_c,N}(g)$ converges in law to $\nu_{\beta_c}(fg)$. Observe that for any $M\geq 1$
 $$
 e^{f_M-f_N}\mu_{\beta_c,N}(dx)=\big(e^{f_M(x) +\beta_c G_{M,4}(x)-\frac{\beta_c^2}{2}\E( G_{M,4}(x)^2)}\big)\nu_{\beta_c,M, N}(dx).
 $$
 On the right hand side the first factor is a random continuous function, independent of the measure $\nu_{\beta_c,M, N}(dx)$,
 which in turn converges in distribution as $N\to\infty.$ A simple argument that employs conditioning (i.e. Fubini) then shows  that the full product on right hand side converges in distribution, whence the same is true for the left hand side.
 As $\sup_{N\geq M}\|f_M-f_N\|_{L^\infty (0,1)}\to 0$ in probability as $M\to\infty$, it is then an easy matter to verify that $\mu_{\beta_c,N}(dx)$ converges in distribution as $N\to\infty.$

\end{proof}

\begin{remark} \label{rem:mesoscopic}{\rm 
A classical results of Selberg yields that   actual logarithm of the Riemann zeta has point-wise Gaussian statistical behaviour. 
Bourgade \cite{Bourgade} has some partial results on joint distributions, and they seem to indicate that in a suitable 'mesoscopic' scaling the random translates of $\log \zeta (1+it)$  behave like a logarithmically correlated Gaussian field. We sketch here how an exact counterpart of this can be shown for the statistical model that we are considering here.  Observe first that  by lemma \ref{le:covariance_approximation} and we may choose a sequence $\varepsilon_n\to 0^+$ and $\lambda_n\to\infty$ so that the covariance of $G_N(\varepsilon_n x)$ satisfies
$$
K_{G_N(\varepsilon \cdot)}(x,y)= \frac{1}{2}\min( \log(1/|x-y|),\log\log N)+\lambda_n+ {\mathcal O}(1),
$$
and, uniformly outside the diagonal, one even has
$$
K_{G_N(\varepsilon \cdot)}(x,y)= \frac{1}{2} \log(1/|x-y|)+\lambda_n +{\mathit o}((1).
$$
On the other hand, we know that our error term $E_n$ converges uniformly to a bounded continuous function. Thus, in the scaling $x\to \varepsilon_n x$ we may write
$$
\mu_{\beta_c,N}(\varepsilon_ndx)\sim  e^{\sqrt{\lambda_n} G_0+R+{\mathit o}((1)}\widetilde\mu_{n,\beta},
$$
where $G_0$ is a fixed standard normal random variable, independent from each $\widetilde \mu^{ n,\beta }$, $R:=E(0)$ is a random variable, and 
$\widetilde\mu_{n,\beta}$ is  obtained by exponentiating  a good approximation of a Gaussian field with the strictly logarithmic covariance structure $\log(1/(x-y|)$ on $[0,1].$ In particular,  $\widetilde \mu_{ n,\beta }$ converges to a standard Gaussian multiplicative chaos on $[0,1]$. Similar statement holds also true in the case $\beta=\beta_c$.}
\end{remark}

\section{Proof of Proposition \ref{le2}: Gaussian approximation of a sum of independent random variables}

We'll start our analysis by considering some general facts about coupling random variables and then apply these facts to Gaussian approximation. Perhaps some of these observations could be found in the literature, and they are far from optimal but we do not need more for our immediate purposes.

Assume that we are given two Borel probability measures $\mu,\nu$ on a metric space $(M,d)$. We may ask how to minimize $\E |X-Y|^p$
over all random variables $X,Y$ taking values in $M$ such that the distribution of $X$ (resp. $Y$) equals $\mu$ (resp. $\nu$).
We denote the infimum of $(\E d(X,Y)^p)^{1/p}$ by $W_p(\mu,\nu)$ (the Wasserstein $p$-distance), and shall restrict ourselves to the case $p=1$.
By denoting by $\gamma$ the joint distribution of $(\mu,\nu)$ on $M\times M$ we see that
$$
W_1(\mu,\nu)=\inf_\gamma\int_{M\times M}d(x,y)\; \gamma (dx\times dy),
$$
where the admissible $\gamma$:s have $\mu$ and $\nu$ as marginals. We start with a simple observation

\begin{lemma}\label{le11}
In the above situation one has that
$$
W_1(\mu,\nu)\leq \inf_{R>0,\; x_0\in M} \Big(4R|\mu-\nu|(B(x_0;R))+32\int_{R/2}^\infty |\mu-\nu|(B(x_0,r)^c)dr.
\Big)$$
\end{lemma}
\begin{proof}
Observe that 
$$
\beta:=\mu-(\mu-\nu)_+=\nu-(\nu-\mu)_+\geq 0.
$$
and define the measure $\beta_\Delta$ on $M\times M$ by $\beta_\Delta (A)=\beta (\{ x\; :\; (x,x)\in A \})$ and note that the measure 
$$
\beta_\Delta +\frac{2}{\|\mu-\nu\|_{TV}}(\mu-\nu)_+\times (\nu-\mu)_+
$$
has the right marginals since $\mu$ and $\nu$ are probability measures so  $\|(\mu-\nu)_+\|_{TV} =\|(\nu-\mu)_+\|_{TV}=(1/2) \|\mu-\nu\|_{TV}$, and both of the marginals of $\beta_\Delta$ are simply $\beta$. As $\beta_{\Delta}$ lives on the diagonal, it follows that
\begin{eqnarray}\label{40}
W_1(\mu,\nu)&\leq &\frac{2}{\|\mu-\nu\|_{TV}}\int_{M\times M}d(x,y)\; (\mu-\nu)_+\times (\nu-\mu)_+ (dx\times dy)\\
&\leq &\frac{2}{\|\mu-\nu\|_{TV}}\int_{M\times M}d(x,y)\; |\mu-\nu|\times |\nu-\mu| (dx\times dy).\nonumber
\end{eqnarray}
Fix now some $x_0\in M$ and $R>0$ and split the integral into ones over $B(x_0,R)\times B(x_0,R)$ and its complement. The integral over $B(x_0,R)\times B(x_0,R)$ we can estimate by noting that here $d(x,y)\leq 2R$ so 
\begin{eqnarray}\label{41}
&&\frac{2}{\|\mu-\nu\|_{TV}}\int_{B(x_0,R)\times B(x_0,R)}d(x,y)\; |\mu-\nu|\times |\nu-\mu| (dx\times dy)\\
&\leq& {2\cdot 2R}{\|\mu-\nu\|_{TV}}^{-1} |\mu-\nu|\times |\nu-\mu| (B(x_0,R)\times B(x_0,R))\nonumber
\leq 4R  |\nu-\mu| (B(x_0,R))
\end{eqnarray}
By symmetry, the integral over the rest has the upper bound
\begin{eqnarray}\label{42}
&&\frac{4}{\|\mu-\nu\|_{TV}}\int_{d(x,x_0)\geq d(y,x_0)\vee R}d(x,y)\; |\mu-\nu|\times |\nu-\mu| (dx\times dy)\\
&\leq&\frac{8}{\|\mu-\nu\|_{TV}}\int_{d(x,x_0)\geq R}d(x,x_0)\; |\mu-\nu|\times |\nu-\mu| (dx\times dy)\nonumber\\ 
&\leq&8\int_{d(x,x_0)\geq  R}d(x,x_0)\; |\mu-\nu| (dx)\nonumber\\ 
&\leq&8\sum_{k=1}^\infty2kR\; \Big(|\mu-\nu|(B(x_0,kR)^c)-|\mu-\nu|(B(x_0,(k+1)R)^c)\Big)\nonumber\\ 
&\leq&16R\sum_{k=1}^\infty\; |\mu-\nu|(B(x_0,kR)^c)\leq 32\int_{R/2}^\infty  |\mu-\nu|(B(x_0,r)^c)dr\nonumber
\end{eqnarray}
The claim follows by combining the estimates \eqref{40}--\eqref{42}.
\end{proof}

We denote by $\widehat\mu$ the Fourier transform) of the measure $\mu$ on $\R^d$ (i.e.  the characteristic function of a random variable with distribution $\mu$).
\begin{corollary}\label{co12} Assume that $\mu$ and $\nu$ are absolutely continuous measures on $\R^d$. Then
$$
W_1(\mu,\nu)\leq \inf_{R\geq 1} C_d\Big(R^{d+1}\|\widehat\mu-\widehat\nu\|_{L^1(\R^d)}+\int_{R/2}^\infty (\mu+\nu)(B(0,r)^c)dr.
\Big)$$
\end{corollary}

\begin{proof}
Let $f$ (resp. $g$) stand for  the density of $\mu$ (resp. $\nu$). The desired statement  follows from the previous lemma as soon as we observe that 
\begin{eqnarray*}
&&\int_{B(0,R)}|f(x)-g(x)|dx\leq C_dR^d\|f-g\|_{L^\infty(\R^d)}\leq C''_dR^d\|\widehat f-\widehat g\|_{L^1(\R^d)}.
\end{eqnarray*}
\end{proof}

Finally, we are ready for:

\begin{proof}[Proof of Proposition \ref{le2}]
All the unspecified constants  (and  the $O(\cdot)$ terms) in the proof are universal in the sense that they may depend only on the the quantities  $d,b_0, b_1,b_2,b_3$. We let $C_j={\rm Cov}(H_j)$ stand for the covariance matrix of the variable $H_j$. 
Denote $\ell_n:= (\sum_{j=1}^n c_j)^{1/2}$  and observe that 
$$
b_0^{-1/2}n^{1/2}\leq\ell_n\leq b_0^{1/2}n^{1/2}.
$$ 
Moreover, set
$$
\displaystyle W:=\ell_n^{-1}\sum_{j=1}^n H_j,
$$
so that ${\rm Tr}({\rm Cov}(W))=d$. By considering instead the random variables $RH_j$ where $R:\R^d\to\R^d$ is a rotation matrix chosen so that $R{\rm Cov}(W)R^{\rm T}$ is diagonal, we may assume that $A:={\rm Cov}(W)$ is diagonal:
$$
A={\rm Cov}(W)=\begin{bmatrix}
\lambda_1&0&0& \ldots &0\\
0&\lambda_2&0&\ldots &0\\
\vdots&&&&\vdots\\
0&\ldots&&0&\lambda_d
\end{bmatrix}
\qquad \textrm{where}\quad \lambda_1\geq\lambda_2\geq\ldots \geq \lambda_d\geq 0\quad \textrm{and}\quad
\sum_{k=1}^d\lambda_j=d.
$$

We start by proving an estimate of  type \eqref{e30} by first assuming that the smallest eigenvalue of $A$ satisfies the lower bound $\lambda_d\geq n^{-2\delta} $, where the constant $\delta\in [0,1/6)$ will be chosen later on. 
Towards that goal, we note that the exponential moment bound \eqref{eq:exp} for $H_k$:s implies that   $\|D^m\varphi_{H_j}\|_{L^\infty (\R^d)}\leq C$ for $m=1,2,3$ and all $j=1,\ldots ,n$, where $\varphi_{H_j}$ stands for the  characteristic function of the variable $H_j$.  Also, we have $D^2\varphi_{H_j}(0)= -{\rm Cov}(H_j)$, whence
$$
\varphi_{H_j}(\xi)=1-\frac{1}{2}\xi^{\rm T}{\rm Cov}(H_j)\xi+ O(|\xi|^3) \qquad \textrm{for all}\;\; \xi.
$$
Hence for the branch of the logarithm that takes value 0 at the point 1 we have for a universal $r_1>0$
\begin{equation}\label{51}
\log \varphi_{H_j}(\xi) = -\frac{1}{2}\xi^{\rm T}{\rm Cov}(H_j)\xi+ O(|\xi|^3) \qquad \textrm{for }\;\; |\xi|\leq 2r_1
\end{equation}
By independence (and since $b_0^{-1}\leq c_j\leq b_0$ for all $j$) we gather that for another universal $r_2>0$
\begin{equation}\label{52}
\log\big(\varphi_{W}(\xi)\big) = \sum_{j=1}^n\log\big(\varphi_{H_j}(\xi/\ell_n)\big) = -\frac{1}{2}\xi^{\rm T}A\xi+ n^{-1/2}O(|\xi|^3) \qquad \textrm{for }\;\; |\xi|\leq r_2\sqrt{n}.
\end{equation}
We note that $\lambda_1\geq 1$ and each $\lambda_j\geq n^{-2\delta}.$ Hence, as $|\xi|^3\approx \sum_{k=1}^d|\xi_k|^3,$ we may estimate component wise and deduce (by also decreasing $r_2$ universally, if needed) 
\begin{equation}\label{53}
|\varphi_{W}(\xi)|\leq \exp\Big( -\frac{1}{4}\xi^{\rm T}\widetilde A\xi\Big)\qquad \textrm{for }\;\;|\widetilde A^{-1}\xi |\leq r_2\sqrt{n},
\end{equation}
where $\widetilde A$ is the $d\times d$ diagonal matrix
$$
\widetilde A:= {\rm diag\,}(1, n^{-2\delta},\ldots, n^{-2\delta})\leq A.
$$

We next choose  a $d$-dimensional centred Gaussian $G$ (independent from the $H_j$:s) such that  
\begin{equation}\label{605}
B:={\rm Cov\,}(G)=r_2^{-2}\log^2(n)\; {\rm diag\,}(n^{-1}, n^{4\delta-1},\ldots, n^{4\delta-1})= (r_2^{-1} \log(n)n^{-1/2}\widetilde A^{-1})^{2}.
\end{equation}
and set
$$
\widetilde W:= G+W.
$$
Then $\varphi_{ \widetilde W}(\xi)=\varphi_{W}(\xi)\exp \big(-\frac{1}{2}\xi^{\rm T}B\xi\big)$ and we estimate
\begin{eqnarray}\label{61}
&&\| \exp(-\frac{1}{2}\xi^{\rm T}A\xi)-\varphi_{\widetilde W}(\xi)\|_{L^1(\R^d)}\\
&= &\Big(\int_{|\widetilde A^{1/2}\xi|\leq \log n}+\int_{\left\{\begin{array}{l}{|\widetilde A^{1/2}\xi|> \log n}\\
{|B^{1/2}\xi|\leq \log n}\end{array}\right.}+\int_{\left\{\begin{array}{l}{|\widetilde A^{1/2}\xi|> \log n}\\
{|B^{1/2}\xi|> \log n}\end{array}\right.}\Big)\big|\exp(-\frac{1}{2}\xi^{\rm T}A\xi)-\varphi_{\widetilde W}(\xi)\big|\; d\xi\nonumber\\
&=& T_1+T_2+T_3.\nonumber
\end{eqnarray}

We make use of the following simple observation for $d\times d$ symmetric matrices $D$ that are  lower bounded by $n^{-\alpha}$ (i.e. by  $n^{-\alpha}I$, where $I$ is the identity matrix) with $\alpha >0$:
\begin{equation}\label{6051}
\textrm{If  }\;\,D\geq n^{-\alpha},\;\;\textrm{then}\quad
\int_{|D^{1/2}\xi|\geq\log n}e^{-\frac{1}{4}\xi^{\rm T}D\xi}d\xi =O(n^{-1/2}).
\end{equation}
Namely,
\begin{eqnarray*}
&&\int_{|D^{1/2}\xi|\geq\log n}e^{-\frac{1}{4}\xi^{\rm T}D\xi}d\xi = |D|^{-1/2}\int_{\|\xi|\geq\log(n)}e^{-|\xi|^2/4}d\xi
\;\lesssim \;n^{d\alpha/2}\int_{r\geq \log(n)}e^{-r^2/4}r^{d-1}dr\\&\lesssim& n^{d\alpha/2}\int_{r\geq \log(n)}e^{-r^2/8}dr
=O\big(n^{d\alpha/2}e^{-\frac{1}{8}\log^2(n)}\big) =O(n^{-1/2})
\end{eqnarray*}
Towards estimating the first term $T_1$ we observe that since $\delta<1/6,$ we have 
$$
\sup_{\{|\widetilde A^{1/2}\xi|\leq \log n\}}n^{-1/2}|\xi|^3 =o(1)\quad \textrm{as}\quad n\to\infty.
$$ Hence we may apply \eqref{52}, the ordering $A\geq \widetilde A$ and the inequality $|e^x-1|\leq 2|x|$ for $x\in (-\infty, 1]$ to obtain the bound 
\begin{eqnarray}\label{62}
T_1&\leq& \int_{| \widetilde A^{1/2}\xi|\leq \log n}e^{-\frac{1}{2}\xi^{\rm T}A\xi}\bigg|\exp \Big(-\frac{1}{2}\xi^{\rm T}B\xi+  n^{-1/2}O(|\xi|^3)\Big)-1\bigg|\; d\xi\\
&\leq& 2\int_{|\widetilde A^{1/2}\xi|\leq \log n}e^{-\frac{1}{2}\xi^{\rm T}\widetilde A\xi}\Big(\frac{1}{2}\xi^{\rm T}B\xi +  n^{-1/2}O(|\xi|^3)\Big)\; d\xi\nonumber\\
&\leq& 2|\widetilde A^{-1/2}|\int_{\R^d}e^{-|\xi|^2/2}\Big(\|\widetilde A^{-1/2}B\widetilde A^{-1/2}\||\xi|^2 + \|\widetilde A^{-1/2}\|^{3} n^{-1/2}O(|\xi|^3)\Big)\; d\xi\nonumber\\
&\lesssim & n^{(d-1)\delta}\big(n^\delta \log^2(n)n^{-1+4\delta} n^\delta+ n^{-1/2}n^{3\delta}\big)\int_{\R^d}e^{-|\xi|^2/2}\big(|\xi|^2 + |\xi|^3\big)\; d\xi\nonumber\\
&=& O\big(n^{-1/2+(d+2)\delta}(n^{3\delta-1/2}\log^2 n+1)\big)\nonumber\\
&=& O\big(n^{-1/2+(d+2)\delta}\big)\nonumber,
\end{eqnarray}
since $\delta<1/6$. Next, by the last equality in \eqref{605}, the condition $|B^{1/2}\xi|\leq \log(n)$ is equivalent to $|\widetilde A^{-1}\xi|\leq r_2n^{1/2}$. Then \eqref{53} and the estimate \eqref{6051} yield
\begin{eqnarray}\label{63}
T_2&\leq& \int_{|\widetilde A^{1/2}\xi|>\log n}\big(e^{-\frac{1}{4}\xi^{\rm T}\widetilde A\xi}+e^{-\frac{1}{2}\xi^{\rm T}A\xi}\big)|\;\;\lesssim\;\; \int_{|\widetilde A^{1/2}\xi|>\log n}e^{-\frac{1}{4}\xi^{\rm T}\widetilde A\xi} = O(n^{-1/2}).
\end{eqnarray}
Finally, for the remaining term $T_3$ we can again invoke \eqref{6051} to obtain
\begin{eqnarray}\label{64}
T_3&\leq& \int_{\left\{\begin{array}{l}{|\widetilde A^{1/2}\xi|> \log n}\\
{|B^{1/2}\xi|> \log n}\end{array}\right.}\Big(\exp(-\frac{1}{2}\xi^{\rm T}A\xi)+\exp(-\frac{1}{2}\xi^{\rm T}B\xi)\Big)\; d\xi\\
&\leq&\int_{|\widetilde A^{1/2}\xi|> \log n}\exp(-\frac{1}{2}\xi^{\rm T}\widetilde A\xi)d\xi\; +\;\int_{|B^{1/2}\xi|> \log n}\exp(-\frac{1}{2}\xi^{\rm T}B\xi)d\xi\; =\;O(n^{-1/2})\nonumber
\end{eqnarray}
Combining  the estimates \eqref{62}--\eqref{64} with \eqref{61} we obtain that
\begin{eqnarray}\label{65}
\| e^{-|\xi|^2/2}-\varphi_{\widetilde W}(\xi)\|_{L^1(\R^d)} = O(n^{-1/2+(d+2)\delta})
\end{eqnarray}

By Bernstein's inequality (a simple application of H\"older's inequality reduces things from the $d$-dimensional case to the one-dimensional one and then one can make use of \cite[Theorem 2.1]{Rio} to get bounds on the tail of the distribution from which one easily gets a Gaussian bound by elementary arguments - the use of \cite[Theorem 2.1]{Rio} is justified by the bounds on the exponential moments)
we have universal constants $n_0,b_4$ such that for $n\geq n_0$ it holds that
\begin{equation}\label{eq:bernstein}
\E \exp(\lambda |W|) \;\leq \;\exp(b_4\lambda^2) \quad\textrm{for all}\quad n\quad \textrm{and for all}\quad
\lambda \leq b_5n^{1/2}.
\end{equation}
Choosing e.g. $\lambda =3$ here and combining  with the excellent Gaussian tail (better than $\lesssim e^{-|\xi|^2/4}$) for $G$ we see that $\Prob (|\widetilde W|>\lambda)<b_5\exp (-2\lambda).$ This yields for $R\geq 1$
estimate
\begin{eqnarray}\label{66}
\int_{R/2}^\infty \big(\Prob(|\widetilde W|\geq r)\big) dr =O(e^{-R})
\end{eqnarray}

We are now ready to invoke Corollary \ref{co12} in combination with \eqref{65} and \eqref{66} in order to deduce the existence of a Gaussian random variable $U$ such that
$$
\E |U-\widetilde W|\lesssim \inf_{R\geq 1}\big( R^{d+1}n^{-1/2+(d+2)\delta} +e^{-R}\big)\lesssim \log^{d+1}(n) n^{-1/2+(d+2)\delta} .
$$
This yields our basic estimate  
\begin{eqnarray}\label{1111}
\E |V|&=& \E |U- W|\leq \E |U-\widetilde W|+\E |G|\,\lesssim\; \log^{d+1}(n) n^{-1/2+(d+2)\delta}+ \log(n)n^{-1+4\delta}\\
&=& O\big(\log^{d+1}(n) n^{-1/2+(d+2)\delta}\big).\nonumber
\end{eqnarray}

We next see how to infer  from \eqref{1111} the inequality \eqref{e30} or \eqref{e30'} in the different cases. For  part (ii) of the Proposition (which also covers the case $d=1$) we may choose $\delta=0$ in \eqref{1111} and obtain directly  
 \eqref{e30'} with $\beta=1/2.$  In order to deal with part (i) of the Proposition (where $d\geq 2$) we assume first that $\lambda_j\geq n^{-(4d+6)^{-1}}.$ Then we may apply directly \eqref{1111} with the choice $\delta=n^{-(2d+3)^{-1}}$ and obtain the inequality \eqref{e30} with the exponent
 $$
 \beta= -1/2+(d+2)(4d+6)^{-1}>0
 $$ 
 that depends only on $\delta.$
In the remaining case there is $k_0\in\{2,\ldots d-1\}$ so that $\lambda_j\geq n^{-(2d+3)^{-1}}$ but $\lambda_{k_0+1}<n^{-(2d+3)^{-1}}.$ Write $W':=(W_1,\ldots,W_{k_0})$ and $W'':=(W_{k_0+1},\ldots,W_d).$ We may apply the above proof on
 $W'$ and find a $k_0$-dimensional Gaussian approximation $U'$ for $W'$ so that $\E | U'-W' |= 
 O\big(\log^{d+1}(n) n^{-\delta}\big).$ We define   the trivial extension $U'$ to a $d$-dimensional random variable $U$ by setting 
 $U=(U',U'')$, where the components of $U''$ are identically zero. Now 
 $$
 \E |W''|\leq (\E |W''|^2)^{1/2}=(\sum_{k=k_0+1}^d\lambda_k)^{1/2}\lesssim n^{-(4d+6)^{-1}}.
 $$
 Finally,
 $$
 \E |V| \leq \E |W'-U'|+\E |W''|\lesssim \log^{d+1}(n)\big(n^{-1/2+(d+1)(2d+3)^{-1}}+n^{-(4d+6)^{-1}}\big) \lesssim\log^{d+1}(n)n^{-(4d+6)^{-1}},
 $$
 where the exponent\footnote{We have not striven to optimality in Proposition \ref{le2} since the obtained bound suffices for the type of applications we have in mind.}   again depends only on $d$.
This proves the desired estimate \eqref{e30}.

We turn to estimating   the exponential moments. Their proof is based solely on \eqref{e30} and the assumed decay of the random variables, so we do not need to separate different cases as before. By the Bernstein estimate \eqref{eq:bernstein} we obtain
$$
\Prob (|V|\geq u)\leq e^{-2\lambda u}e^{4b_4\lambda^2}\qquad\textrm{for any } \quad u>0\quad\textrm{and}\quad \lambda\in(0,b_4\sqrt{n}).
$$
Denote  $\delta:=a_1n^{-1/2}\log^{d+1}(n)$. Assume that $\lambda\in (0, b_4\sqrt{n})$. We invoke the Bernstein estimate to obtain  (assuming $n$ big enough), for an auxiliary parameter $M\geq 1$
\begin{eqnarray*}
\E e^{\lambda |V|} &=&1+ \E \big( |V|\frac{\exp(\lambda V)-1}{|V|}\chi_{\{|V|\leq M\}}\big) + (e^{\lambda M}-1)\Prob (|V|>M)+\lambda \int_M^\infty e^{\lambda u}\Prob (|V|>u)du.
\end{eqnarray*}
By noting that $t\mapsto t^{-1}(e^{\lambda t}-1)$ (defined to be zero at zero) is increasing on $[0,M]$, and hence less than $M^{-1}(e^{\lambda M}-1)$ on that interval, we deduce
\begin{eqnarray*}
\E \exp(\lambda |V|) -1 &\leq& \delta (e^{\lambda M}-1)M^{-1}+ (e^{\lambda M}-1)e^{-2\lambda M}e^{4b_4\lambda^2-2\lambda M} +e^{4b_4\lambda^2}\lambda \int_M^\infty e^{-Mu}du\\
&\leq& \delta e^{\lambda M}M^{-1}+2e^{-\lambda M}e^{4b_4\lambda^2}
\end{eqnarray*}
The desired estimate is obtained by choosing $M$ so that $\sqrt{\delta}=e^{-\lambda M}$ and plugging in the definition of $\delta.$

Assume then the that variables $h_k$ are uniformly bounded. In this case a standard application of Azuma's inequality yields universal constants $s, r>0$ so that
$$
\Prob (|V|\geq u)\leq  se^{-2ru^2} \quad \textrm{for all}\quad u>0.
$$
In an analogous manner  to what we just did for the exponential moments,  for any $M>0$ it follows that
\begin{eqnarray*}
\E e^{rV^2} &=&1+ \E \big( |V|\frac{e^{rV^2}-1}{|V|}\chi_{\{|V|\leq M\}}\big) + (e^{rM^2}-1)\Prob (|V|>M)+2r\int_M^\infty xe^{rx^2}\Prob (|V|>x)dx
\end{eqnarray*}
and deduce
\begin{eqnarray*}
\E \exp(r|V|^2) &\leq& 1+\delta (e^{rM^2}-1)M^{-1}+ s(e^{rM^2}-1)e^{-2rM^2} +s\int_M^\infty 2rxe^{-rx^2}dx\\
&\leq& 1+\delta M^{-1}e^{rM^2}+2se^{-rM^2}
\end{eqnarray*}
The desired estimate is  obtained by  this time choosing $M$ so that $\sqrt{\delta}=e^{-rM^2}$.
\end{proof}

\end{document}